%% file: LangevinKL-arxiv.tex
\documentclass[11pt,twoside]{article}
\usepackage{fullpage}
\usepackage{epsf}
\usepackage{fancyhdr}
\usepackage{graphics}
\usepackage{graphicx}
\usepackage{psfrag}
\usepackage{microtype}
\usepackage{subfigure}
\usepackage{algorithmic}
\usepackage{color,xcolor}
\usepackage[linesnumbered,ruled]{algorithm2e}
\DontPrintSemicolon
\usepackage{color}
\usepackage{tabularx}
\usepackage{amsthm}
\usepackage{amsfonts}
\usepackage{amsmath}
\usepackage{amssymb,bbm}
\usepackage{stackengine}
\usepackage{footnote}
\makesavenoteenv{tabular}
\makesavenoteenv{table}
\input{final_macros.tex}

\usepackage{natbib}
\usepackage{url}
\usepackage[colorlinks,linkcolor=magenta,citecolor=blue, pagebackref=true,backref=true]{hyperref}
\renewcommand*{\backrefalt}[4]{%
    \ifcase #1 \footnotesize{(Not cited.)}%
    \or        \footnotesize{(Cited on page~#2.)}%
    \else      \footnotesize{(Cited on pages~#2.)}%
    \fi}

\usepackage{nicefrac}
\usepackage{comment}

\usepackage{chngpage}

 \usepackage{tabularx}%

\usepackage{enumitem}
\usepackage{booktabs}
\usepackage{caption}

\usepackage{bm,bbm}
\usepackage{mathtools}

\newtheorem{assumption}{Assumption}
\newcommand{\Wass}{\ensuremath{\mathcal{W}}}
\makeatletter
\long\def\@makecaption#1#2{
        \vskip 0.8ex
        \setbox\@tempboxa\hbox{\small {\bf #1:} #2}
        \parindent 1.5em  
        \dimen0=\hsize
        \advance\dimen0 by -3em
        \ifdim \wd\@tempboxa >\dimen0
                \hbox to \hsize{
                        \parindent 0em
                        \hfil 
                        \parbox{\dimen0}{\def\baselinestretch{0.96}\small
                                {\bf #1.} #2
                                } 
                        \hfil}
        \else \hbox to \hsize{\hfil \box\@tempboxa \hfil}
        \fi
        }
\makeatother

\begin{document}

\begin{center}
{\bf{\LARGE{Improved Bounds for Discretization of Langevin Diffusions: Near-Optimal Rates without Convexity}}}

\vspace*{.2in}
 \large{
 \begin{tabular}{cccc}
  Wenlong Mou$^{ \diamond}$ & Nicolas Flammarion$^{ \diamond}$ &
  Martin J. Wainwright$^{\dagger, \diamond, \ddagger}$ &Peter L.
  Bartlett$^{\diamond, \dagger}$
 \end{tabular}

}

\vspace*{.2in}

 \begin{tabular}{c}
 Department of Electrical Engineering and Computer Sciences$^\diamond$\\
 Department of Statistics$^\dagger$ \\
 UC Berkeley\\
 \end{tabular}

 \vspace*{.1in}
 \begin{tabular}{c}
 The Voleon Group$^\ddagger$
 \end{tabular}

\vspace*{.2in}

\today

\vspace*{.2in}

\begin{abstract}%
  We present an improved analysis of the Euler-Maruyama discretization
  of the Langevin diffusion. Our analysis does not require global
  contractivity, and yields polynomial dependence on the time horizon.
  Compared to existing approaches, we make an additional smoothness
  assumption, and improve the existing rate from $O(\eta)$ to
  $O(\eta^2)$ in terms of the KL divergence. This result matches
  the correct order for numerical SDEs, without suffering from
  exponential time dependence.  When applied to algorithms for
  sampling and learning, this result simultaneously improves all those
  methods based on Dalayan's approach.
\end{abstract}

\end{center}

\section{Introduction}

In recent years, the machine learning and statistics communities have
witnessed a surge of interest in the Langevin diffusion process, and
its connections to stochastic algorithms for sampling and
optimization.  The Langevin diffusion in $\real^d$ is defined via the
It\^o stochastic differential equation (SDE)
\begin{align}
\label{eq:langevin-equation}
 d X_t = b(X_t) dt + dB_t,
\end{align}
where $B_t$ is a standard $d$-dimensional Brownian motion, and the
function $b: \real^d \to \real^d$ is known as the drift term.  For a
drift term of the form $b(x) = -\frac{1}{2}\nabla U(x)$ for some
differentiable function $U: \real^d \rightarrow \real$, the Langevin
process~\eqref{eq:langevin-equation} has stationary distribution with
density $\gamma(x) \propto e^{-U(x)}$; moreover, under mild growth
conditions on $U$, the diffusion converges to this stationary
distribution as $t \rightarrow \infty$. See~\cite{Pav14} for more
background on these facts, which underlie the development of sampling
algorithms based on discretizations of the Langevin diffusion.
Diffusive processes of this nature also play an important role in
understanding stochastic optimization; in this context, the Gaussian
noise helps escaping shallow local minima and saddle points in finite
time, making it especially useful for non-convex optimization.  From a
theoretical point of view, the continuous-time process is attractive
to analyze, amenable to a range of tools coming from stochastic
calculus and Brownian motion theory~\citep{MR1725357}. However, in
practice, an algorithm can only run in discrete time, so that the
understanding of discretized versions of the Langevin diffusion is
very important.

The discretization of SDEs is a central topic in the field of
scientific computation, with a wide variety of schemes proposed and
studied~\citep{platen,higham2001algorithmic}.  The most commonly used
discretization is the Euler-Maruyama discretization: parameterized by
a step size $\stepsize > 0$, it is defined by the recursion
\begin{align}
\label{eq:langevin-euler}
\discretized{X}_{(k+1)\stepsize} = \discretized{X}_{k \stepsize} + \stepsize b(\discretized{X}_{k \stepsize}) +
\sqrt{\stepsize} \xi_k, \quad \mbox{for $k= 0, 1, 2, \ldots$.}
\end{align}
Here the sequence $\{\xi_k \}_{k=1}^{+\infty}$ is formed of
$\mathrm{i.i.d.}$ $d$-dimensional standard Gaussian random vectors.

From past work, the Euler-Murayama scheme is known to have first-order
accuracy under appropriate smoothness conditions.  In particular, the
Wasserstein distances $\Wass_p$ for $p \geq 1$
between the original Langevin
diffusion and the discretized version decays as $O(\stepsize)$ as
$\stepsize$ decays to zero, with the dimension $d$ and time horizon
$T$ captured in the order notation~\citep[see,
  e.g.,][]{alfonsi2014optimal}.  When the underlying dependence on the
time horizon $T$ is explicitly calculated, it can grow exponentially,
due to the underlying \mbox{Gr\"{o}nwall} inequality. If the potential
$U$ is both suitably smooth and strongly convex, then the scaling with
$\stepsize$ remains first-order, and the bound becomes independent of
time $T$~\citep{durmus2017nonasymptotic,dalalyan2017user}. These
bounds, in conjunction with the coupling method, have been used to
bound the mixing time of the unadjusted Langevin algorithm (ULA) for
sampling from strongly-log-concave densities. Moreover, this bound
aligns well with the classical theory of discretization for ordinary
differential equations (ODEs), where finite-time discretization error
may suffer from bad dependence on $T$, and either contraction
assumptions or symplectic structures are needed in order to control
long-time behavior~\citep{iserles2009first}.

On the surface, it might seem that SDEs pose greater numerical
challenges than ODEs; however, the presence of randomness actually has
been shown to help in the long-term behavior of discretization.
\citet{dalalyan2017theoretical} showed that the pathwise
Kullback-Leibler (KL) divergence between the original Langevin
diffusion~\eqref{eq:langevin-equation} and the Euler-Maruyma
discretization~\eqref{eq:langevin-euler} is bounded as $O(\stepsize
T)$ with only smoothness conditions. This result enables comparison of
the discretization with the original diffusion over long time
intervals, even without contraction.  The discretization techniques
of~\citet{dalalyan2017theoretical} serve as a foundation for a number
of recent papers on sampling and non-convex learning, including the
papers~\citep{raginsky2017non,tzen2018local,liang2017statistical}.

On the other hand, this $O(\stepsize)$ bound on the KL error is likely to
be loose in general. Under suitable smoothness conditions, standard
transportation
inequalities~\citep{bolley2005weighted}
guarantee that such a KL bound can be translated into a
$O(\sqrt{\stepsize})$-bound in Wasserstein distance. Yet, as mentioned in
the previous paragraph, the Wasserstein rate should be $O(\stepsize)$ under
enough smoothness assumption. This latter result either requires
assuming contraction or leads to exponential time dependence, leading
naturally to the question: can we achieve best of both worlds?
That is, is it possible to prove a $O(\stepsize T)$ Wasserstein bound
without convexity or other contractivity conditions?

\paragraph{Our contributions:}
In this paper, we answer the preceding question in the affirmative:
more precisely, we close the gap between the correct rate for the
Euler-Maruyama method and the linear dependence on time horizon,
without any contractivity assumptions. As long as the drift term
satisfies certain first and second-order smoothness, as well as growth
conditions at far distance, we show the KL divergence between marginal
distributions of equation~\eqref{eq:langevin-equation} and
equation~\eqref{eq:langevin-euler}, at any time $T$, is bounded as
$O(\stepsize^2 d^2 T)$. Note that this bound is non-asymptotic, with
polynomial dependence on all the smoothness parameters, and linear
dependence on $T$.  As a corollary of this improved discretization
bound, we give improved bounds for using the unadjusted Langevin
algorithm (ULA) for sampling from a distribution satisfying a
log-Sobolev inequality.  In addition, our improved discretization
bound improves a number of previous results on non-convex optimization
and inference, all of which are based on the discretized Langevin
diffusion.

In the proof of our main theorem, we introduce a number of new
techniques.  A central challenge is how to study the evolution of time
marginals of the interpolation of discrete-time Euler algorithm, and
in order to do so, we derive a Fokker-Planck equation for the
interpolated process, where the drift term is the backward conditional
expectation of $b$ at the previous step, conditioned on the current
value of $x$. The difference between this new drift term for the
interpolated process and $b$ itself can be much smaller than the
difference between $b$ at two time points.  Indeed, taking the
conditional expectation cancels out the bulk of the noise terms,
assuming the density from the previous step is smooth enough. We
capture the smoothness of density at the previous step by its Fisher
information, and develop De Bruijn-type inequalities to control the
Fisher information along the path. Combining this regularity estimate
with suitable tail bounds leads to our main result. We suspect that
our analysis of this interpolated process and associated techniques
for regularity estimates may be of independent interest.


\begin{table}[t]
  \centering {\small{
    \begin{tabular}{|c|c|c|c|c|c|}
    \hline Paper& \stackanchor{Require}{contraction}& \footnote{We
      only listed time horizon dependence for methods that guarantee
      discretization error between continuous-time and discrete-time
      for any time. If the proof requires mixing and does not give the
      difference between the one-time distributions, we mark it as
      ``-''.}Time $T$& \footnote{The distances are measured in
      $\Wass_p$. If the original bound is shown for
      KL, it is transformed into $\Wass_p$ using transportation inequalities,
      resulting in the same rate. We mark with * if the original
      bound was shown in KL}
       Step size $\stepsize$ &
    \stackanchor{Require}{mixing} &
    \stackanchor{Additional}{assumptions} \\ \hline
    \cite{dalalyan2017theoretical}& No& $O(\sqrt{T})$&$O(\sqrt{\stepsize})$
    &No&None \\ \hline
    \cite{alfonsi2014optimal}&No&$O(e^{cT})$&$O(\stepsize)$&No&\stackanchor{Second-order}{smooth
      drift}\\ \hline
    \stackanchor{\cite{dalalyan2017user},}{\cite{durmus2017nonasymptotic}}&Yes&
    -&$O(\stepsize)$&Yes&\stackanchor{Second-order}{smooth drift}\\ \hline
    \stackanchor{\cite{cheng2017convergence},}{\cite{ma2018sampling}}&No&-&$O(\sqrt{\stepsize})$
    &Yes&\stackanchor{strong convexity}{outside a ball}\\ \hline
    \stackanchor{\cite{cheng2018sharp},}{\cite{bou2018coupling}}\footnote{For
      Hamiltonian Monte-Carlo, which is based on discretization of ODE, instead of SDE.}&No&-&$O(\stepsize)$&Yes&\stackanchor{strong
      convexity}{outside a ball}\\ \hline This
    paper&No&$O(\sqrt{T})$&$O(\stepsize)$
    &No&\stackanchor{Second-order}{smooth drift}\\ \hline
    \end{tabular}
    }}
    \caption{Comparison between discretization of Langevin diffusion
      and sampling algorithms.}
    \label{tab:comparison}
\end{table}

\paragraph{Related work:} Recent years have witnessed a flurry of activity in
statistics and machine learning on the Langevin diffusion and related
stochastic processes.  A standard application is sampling from a
density of the form $\gamma(x) \propto e^{-U(x)}$ based on an oracle
that returns the pair $( U(x), \nabla U(x))$ for any query point $x$.
In the log-concave case, algorithms for sampling under this model are
relatively well-understood, with various methods for discretization
and variants of Langevin diffusion proposed in order to refine the
dependence on dimension, accuracy level and condition
number~\citep{dalalyan2017theoretical,durmus2017nonasymptotic,cheng2017underdamped,lee2018algorithmic,mangoubi2018dimensionally,dwivedi2018log}.

When the potential function $U$ is non-convex, the analysis of
continuous-time convergence and the discretization error analysis both
become much more involved. When the potential satisfies a logarithmic
Sobolev inequalities, continuous-time convergence rates can be
established~\citep[see e.g.][]{markowich2000trend}, and these
guarantees have been leveraged for sampling
algorithms~\citep{bernton2018langevin,wibisono2018sampling,
  ma2018sampling}. Coupling-based results for the Wasserstein distance
$\Wass_2$ have also been shown for variants of Langevin
diffusion~\citep{cheng2018sharp,bou2018coupling}. Beyond sampling, the
global convergence nature of Langevin diffusion has been used in
non-convex optimization, since the stationary distribution is
concentrated around global minima. Langevin-based optimization
algorithms have been studied under log-Sobolev
inequalities~\citep{raginsky2017non}, bounds on the Stein
factor~\citep{erdogdu2018global}; in addition, accelerated methods
have been studied~\citep{chen2018accelerating}.  The dynamics of
Langevin algorithms have also been studied without convergence to
stationarity, including exiting times~\citep{tzen2018local}, hitting
times~\citep{zhang2017hitting}, exploration of a
basin-of-attraction~\citep{lee2018beyond}, and statistical inference
using the path~\citep{liang2017statistical}. Most of the works in
non-convex setting are based on the discretization methods introduced
by~\citet{dalalyan2017theoretical}.

Finally, in a concurrent and independent line of work,
\citet{fang2018multilevel} also studied a multi-level sampling
algorithm without imposing a contraction condition, and obtained
bounds for the mean-squared error; however, their results do not give
explicit dependence on problem parameters.  Since the proofs involve
bounding the moments of Radon-Nikodym derivative, their results may be
exponential in dimension, as opposed to the polynomial-dependence
given here.

\paragraph{Notation:} We let $\Vert x \Vert_2$ denote the Euclidean norm of
a vector $x \in \real^d$. For a matrix $M$ we let $\opnorm{M}$ denote
its spectral norm. For a function $b: \real^d \rightarrow \real^d$, we
let $\nabla b(x) \in \real^{d \times d}$ denote its Jacobian evaluated
at $x$. We use $\mathcal{L}(X)$ to denote the law of random variable
$X$. We define the constant $A_0=\Vert b(0) \Vert_2$. When the
variable of the integrand is not explicitly written, integrals are
taking with respect to the Lebesgue measure: in particular, for an
integrable function $g: \real^d \rightarrow \real$, we use $\int g$ as
a shorthand for $\int_{\real^d}g(x)dx $. For a probability density function $p$ (with respect to the Lebesgue measure) in $\real^d$, we let $H (p)$ to denote the differential entropy of $p$.


\section{Main Results}
\label{SecMain}

We now turn to our main results, beginning with our assumptions and a
statement of our main theorem.  We then develop and discuss a number
of corollaries of these main results.

\subsection{Statement of main results}

Our main results involve three conditions on the drift term $b$, and
one on the initialization:
\begin{assumption}[Lipschitz drift term]
\label{assume-lipschitz-drift}
There is a finite constant $\smooth$ such that
\begin{align}
  \Vert b(x)-b(y)\Vert_2 \leq \smooth \Vert x-y \Vert_2 \quad
  \mbox{for all $x,y \in \real^d$.}
\end{align}
\end{assumption}
\begin{assumption}[Smooth drift term]\label{assume-smooth-drift}
There is a finite constant $\hessianlip$ such that
\begin{align}
  \opnorm{\nabla b(x) - \nabla b(y)}\leq \hessianlip \Vert x-y \Vert_2 \quad
  \mbox{for all $x,y \in \real^d$.}
\end{align}
\end{assumption}
\begin{assumption}[Distant dissipativity]
  \label{assume-strong-dissipative}
  There exist strictly positive constants $\mu, \beta$ such that
  \begin{align}
    \langle b(x),x \rangle \leq -\mu \Vert x \Vert_2^2+\beta \quad
    \mbox{for all $x \in \real^d$.}
  \end{align}
\end{assumption}
\begin{assumption}[Smooth Initialization]
  \label{assume-smooth-initialize}  
  The initializations $X_0$ and $\discretized{X}_0$, for the
  processes~\eqref{eq:langevin-equation} and~\eqref{eq:langevin-euler}
  respectively, are drawn from a density $\pi_0$ such that
  \begin{align}
    \label{eq:smooth-initialize}
    -\log \pi_0(x)\leq h_0+\frac{\Vert x \Vert_2^2}{\sigma_0^2} \quad
    \mbox{for all $x \in \real^d$.}
  \end{align}
\end{assumption}
\noindent Note that no contractivity assumption on the drift term $b$
is imposed. Rather, we use the notion of distant dissipativity, which
is substantially weaker; even this assumption is relaxed in
Theorem~\ref{ThmWeakAssumption}. The initialization
condition~\eqref{eq:smooth-initialize} is clearly satisfied by the
standard Gaussian density, but
Assumption~\ref{assume-smooth-initialize} allows for other densities
with quadratic tail behavior. \\

\noindent With these definitions, the main result of this paper is the
following:
\begin{theorem}
\label{ThmMain}
Consider the original Langevin diffusion~\eqref{eq:langevin-equation}
under Assumptions~\ref{assume-lipschitz-drift},
~\ref{assume-smooth-drift}, ~\ref{assume-strong-dissipative}
and~\ref{assume-smooth-initialize}.  Then there are universal
constants $(\unicon_0, \unicon_1)$ such that for any $\stepsize \in
(0, \frac{1}{2 \smooth})$ and all times $T > 0$, the KL error of the
Euler-Maruyama discretization~\eqref{eq:langevin-euler} is bounded as
\begin{multline}
  \label{EqnMainBound}
    \kull{\discretized{\pi}_{T} }{ \pi_T } \leq \unicon_0 \stepsize^2
    \bigg( h_0 + H(\pi_0) + A_0^2 + \Big( \sigma_0^2 d + \frac{\beta + d }{\mu}
    \Big) \Big( \frac{1}{\sigma_0^{2}} + T \smooth^2\Big) + T
    \hessianlip^2 d^2 \bigg) \\
+ \unicon_1 \stepsize^4 \hessianlip^2 \bigg(A_0^4 + \smooth^4
\Big(\sigma_0^2 d + \frac{(\beta+d)^2}{\mu} + d^2 \Big)\bigg).
  \end{multline}

\end{theorem}

If we track only the dependence on $(\stepsize, T, d)$, the
result~\eqref{EqnMainBound} can be summarized as a bound of the form
$\kull{\discretized{\pi}_{T}}{\pi_T} \lesssim \stepsize^2 d^2 T$. This
result should be compared to the $O(\stepsize d T)$ bound obtained
by~\citet{dalalyan2017theoretical} using only
Assumption~\ref{assume-smooth-drift}. It is also worth noticing that
the term $\stepsize^2d^2\hessianlip^2T$ only comes with the third
order derivative bound, which coincides with the Wasserstein distance
result, based on a coupling proof, as obtained
by~\citet{durmus2017nonasymptotic}
and~\citet{dalalyan2017user}. However, these works do not study
separately the discretization error of the discrete process and assume
contractivity.

Note that Assumption~\ref{assume-strong-dissipative} can be
substantially relaxed when the drift is negative gradient of a
function. Essentially, we only require this function to be
non-negative, along with the smoothness assumptions. In such case, we
have the following discretization error bound:

\begin{theorem}
\label{ThmWeakAssumption}
Consider the original Langevin diffusion~\eqref{eq:langevin-equation}
under Assumptions~\ref{assume-lipschitz-drift},
~\ref{assume-smooth-drift}, and~\ref{assume-smooth-initialize}, and
suppose that $b = - \nabla f$ for some non-negative function $f$. Then
for any stepsize $\stepsize \in (0, \frac{1}{2 \smooth})$ and time $T
> 0$, the KL error of the Euler-Maruyama
discretization~\eqref{eq:langevin-euler} is bounded as
\begin{multline}
  \label{EqnMainTwo}
    \kull{\discretized{\pi}_{T} }{ \pi_T } \leq \unicon_0 \stepsize^2
    \bigg( A_0^2 + \Big( \sigma_0^2 d + f(0) + \smooth T \sigma_0^2
    (h_0 + H(\pi_0) + d) \Big) \Big( \frac{1}{\sigma_0^{2}} + T \smooth^2\Big) + T
    \hessianlip^2 d^2 \bigg) \\
+ \unicon_1 \stepsize^4 \hessianlip^2 \bigg(A_0^4 + \smooth^4 \Big(
f(0)^2 + \smooth^2 T^2 \sigma_0^4 (h_0 + d)^2 + \smooth^2 T^4 d^2
\Big)\bigg).
  \end{multline}
\end{theorem}

Once again tracking only the dependence on $(\stepsize, T, d)$, the
bound~\eqref{EqnMainTwo} can be summarized as
$\kull{\discretized{\pi}_T}{\pi_T} \lesssim \stepsize^2 T d(d +
T)$. This bound has weaker dependency on $T$, but it holds for any
non-negative potential function without any growth conditions.

When the problem of sampling from a target distribution $\gamma(x)
\propto e^{-U(x)}$ is considered, the above bounds applied to the
drift term $b(x) = -\frac{1}{2}\nabla U(x)$ yield bounds in TV
distance, more precisely via the convergence of the Fokker-Planck
equation and the Pinsker
inequality~\citep{dalalyan2017theoretical}. Instead, in this paper, so
as to obtain a sharper result, we directly combine the result of
Theorem~\ref{ThmMain} with the analysis
of~\citet{cheng2017convergence}.  A notable feature of this strategy
is that it completely decouples analyses of the discretization error
and of the convergence of the continuous-time diffusion process. The
convergence of the continuous-time process is guaranteed when the
target distribution satisfies a log-Sobolev
inequality~\citep{toscani1999entropy,markowich2000trend}.

Given an error tolerance $\varepsilon>0$ and a distance function
$\dist$, we define the associated \emph{mixing time} of the discretized process
\begin{align}
\label{EqnMixingTime}
N(\varepsilon, \dist) & \defn \arg \min_{k = 1, 2, \ldots} \left \{
\dist(\discretized \pi_{k \stepsize},\gamma) \leq \epsilon \right \}.
\end{align}
With this definition, we have the following:
\begin{corollary}
  \label{cor:main}
  Consider a density of the form $\gamma(x) \propto \exp(-U(x))$ such
  that:
  \begin{enumerate}
\item[(a)] The gradient $\nabla U$ satisfies
  Assumptions~\ref{assume-lipschitz-drift},~\ref{assume-smooth-drift},
  and~\ref{assume-strong-dissipative}.
\item[(b)] The distribution defined by $\gamma$ satisfies a
  log-Sobolev inequality with constant $\rho > 0$.
  \end{enumerate}
Then under the initialization
Assumption~\ref{assume-smooth-initialize}, for any $\varepsilon>0$,
the unadjusted Langevin algorithm~\eqref{eq:langevin-euler} with drift
\mbox{$b = - \frac{1}{2} \nabla U$} and step size
$\stepsize=\frac{\sqrt{\varepsilon \rho}}{d}(\log
\frac{1}{\rho})^{-1}$ has mixing times bounded as:
\begin{align*}
   \begin{cases}
     N(\varepsilon, D_{KL})=
     \tilde{O}\left(\varepsilon^{-1/2}d{\rho}^{-3/2}\right)&\text{in
       KL-divergence}, \\ N(\varepsilon, TV)=\tilde{O}\left(d
     \varepsilon^{-1}\rho^{-\frac{3}{2}}\right)&\text{in TV
       distance},\\ N(\varepsilon, \Wass_2)=\tilde{O}\left(d
     \varepsilon^{-1}\rho^{-\frac{5}{2}}\right)&\text{in
     }\Wass_2 \text{ distance},\\
     N(\varepsilon,\Wass_1)=\tilde{O}\left( d^{\frac{3}{2}}
     \varepsilon^{-1}\rho^{-\frac{3}{2}}\right) &\text{in }\Wass_1 \text{ distance}.
   \end{cases}
 \end{align*}
\end{corollary}
The set of distributions satisfying a log-Sobolev
inequality~\citep{gross1975logarithmic} includes strongly log-concave
distributions~\citep{bakry1985diffusions} as well as perturbations
thereof~\citep{holley1987logarithmic}. For example, it includes
distributions that are strongly log-concave outside of a bounded
region, but non-log-concave inside of it, as analyzed in some recent
work~\citep{ma2018sampling}. Under the additional smoothness
Assumption~\ref{assume-smooth-drift}, we obtain an improved
mixing-time $O(d/\sqrt{\varepsilon})$ compared to $O(d/\varepsilon)$
of \citet{ma2018sampling}. On the other hand, we obtain the same
mixing time in $\Wass_2 \text{ distance}$ as the
papers~\citep{durmus2017nonasymptotic, dalalyan2017user} but under
weaker assumptions on the target distribution---namely, those that
satisfy a log-Sobolev inequality as opposed to being strongly
log-concave.


\subsection{Overview of proof}

In this section, we provide a high-level overview of the three main
steps that comprise the proof of Theorem~\ref{ThmMain}; the subsequent
Sections~\ref{sec:fokker},~\ref{sec:kl}, and~\ref{sec:regu} provide
the details of these steps.

\paragraph{Step 1:}  First,  we construct a
continuous-time interpolation $\discretized X_t$ of the discrete-time
process $\discretized X_{k \stepsize}$, and prove that its density
$\discretized \pi_t$ satisfies an analogue of the Fokker-Plank
equation (see Lemma~\ref{lemma-Fokker-Planck-for-interploated}). The
elliptic operator of this equation is time-dependent, with a drift
term $\discretized b_t = \mathbb{E} (b(\discretized X_{k \stepsize}) |
\discretized X_t =x)$ given by the backward conditional expectation of
the original drift term $b$. By direct calculation, the time
derivative of the KL divergence between the interpolated and the
original Langevin diffusion $\kull{\discretized \pi_t}{\pi_t}$ is
controlled by the mean squared difference between the drift terms of
the Fokker-Planck equations for the original and the interpolated
processes, namely the quantity
\begin{align}
  \label{EqnMSETerm}
  \int \discretized \pi_t(x)\Vert b(x) - \discretized b_t (x)
  \Vert_2^2 dx.
\end{align}
See Lemma~\ref{lemma-kl-time-derivative} for details.
\paragraph{Step 2:}  Our next step is to control the mean-squared
error term~\eqref{EqnMSETerm}.  Compared to the MSE bound obtained
from the Girsanov theorem by~\citet{dalalyan2017theoretical}, our
bound has an additional backward conditional expectation inside the
norm. Directly pulling this latter outside the norm by convexity
inevitably entails a KL bound $O(\stepsize)$ due to fluctuations of the
Brownian motion. However, taking the backward expectation cancels out
most of the noises, as long as the distribution of the initial iterate
at each step is smooth enough. This geometric intuition is explained
precisely in Section~\ref{subsection-geometric-intuition}, and
concretely implemented in Section~\ref{subsection-mse-to-fisher}. The
following proposition summarizes the main conclusion from Steps 1 and
2:
\begin{proposition}
  \label{prop-KL-derivative-fisher-moment}
Under
Assumptions~\ref{assume-lipschitz-drift},~\ref{assume-smooth-drift}
and~\ref{assume-smooth-initialize}, for any $t \in [ k \stepsize, (k+1)
  \stepsize]$, we have 
\begin{multline}
  \label{EqnKLDerivativeFisher}
\frac{d}{dt}\kull{\discretized \pi_t}{ \pi_t } \leq 4 \smooth^2 (t-k \stepsize)^2
\int \discretized{\pi}_{k \stepsize}\Vert \nabla \log \discretized{\pi}_{k \stepsize}\Vert_2^2 +
12 \smooth^4 ( t - k \stepsize)^3d \\
+ 16 (t-k \stepsize)^4 \hessianlip^2 \left(A_0^4 + \smooth^4 \mathbb{E} \Vert
\discretized{X}_{k \stepsize} \Vert_2^4 \right) + 48 (t-k \stepsize)^2 \hessianlip^2 d^2.
\end{multline}
\end{proposition}

\paragraph{Step 3:}  The third step
is to bound the moments of $\nabla \log \discretized \pi_{k \stepsize}$ and $\discretized
X_{k \stepsize}$, so as to control the right-hand side of
equation~\eqref{EqnKLDerivativeFisher}.  In order to bound the Fisher
information term $\int \discretized{\pi}_{k \stepsize}\Vert \nabla \log
\discretized{\pi}_{k \stepsize}\Vert_2^2$, we prove an extended version of the De
Brujin formula for the Fokker-Planck equation of $\discretized \pi_t$ (see
Lemma~\ref{lemma-regularity-time-integral}). It bounds the time
integral of $\int \discretized{\pi}_{t}\Vert \nabla \log
\discretized{\pi}_{t}\Vert_2^2$ by moments of $X_t$. Since
Proposition~\ref{prop-KL-derivative-fisher-moment} requires control of
the Fisher information at the grid points $\{k \stepsize \}_{k \in
  \mathbb{N}}$, we bound the integral at time $k \stepsize$ by the one at
time $t \in[(k-1) \stepsize, k \stepsize]$; see
Lemma~\ref{lemma-regularity-relative-difference} for the precise
statement. Combining these results, we obtain the following bound of
the averaged Fisher information.
\begin{proposition}
 \label{prop-fisher-grid}
    Under
    Assumptions~\ref{assume-lipschitz-drift},~\ref{assume-smooth-drift}
    and~\ref{assume-smooth-initialize}, for $T=N \stepsize$ and $N \in
    \mathbb{N}_+$, we have
    \begin{multline*}
        \frac{1}{N}\sum_{k=1}^N\! \int \!\discretized{\pi}_{k \stepsize}\Vert \nabla
        \log \discretized{\pi}_{k \stepsize}\Vert_2^2 \leq (32h_0+128A_0^2T  + H (\pi_0))\\
        +
        32 \sigma_0^{-2} \Exs \vecnorm{\discretized{X}_T}{2}^2 + 128\smooth^2 \int_0^T   \mathbb{E}\Vert \discretized{X}_t \Vert_2^2dt+32 \stepsize^2d^2\hessianlip^2T.
    \end{multline*}
\end{proposition}
It remains to bound the moments of $\discretized{X}_t$ along the
path. By Proposition~\ref{prop-KL-derivative-fisher-moment} and Proposition~\ref{prop-fisher-grid}, the second and fourth moment of $\discretized{X}_t$ are used. With different assumptions on the drift term, different moments bounds can be established, leading to Theorem~\ref{ThmMain} and Theorem~\ref{ThmWeakAssumption}, respectively.

\begin{itemize}
    \item Under distant dissipativity (Assumption~\ref{assume-strong-dissipative}), the $p$-th moment of this process can be bounded from above, for arbitrary value of $p > 1$. (see
Lemma~\ref{lemma-tail-strong-dissipative}). The proof is based on the
Burkholder-Davis-Gundy inequality for continuous martingales. Collecting these results yields Eq~\eqref{EqnMainBound}, which completes our sketch of the proof of Theorem~\ref{ThmMain}.
    \item Without Assumption~\ref{assume-strong-dissipative}, if the drift term is negative gradient of a function $b = - \nabla f$ with $f \geq 0$, the second and fourth moment can still be bounded (see Lemma~\ref{lemma:moment-24-only-non-neg}). The proof uses the moment bounds of $\nabla f$ along the path.  Collecting these results yields Eq~\eqref{EqnMainTwo}, which completes our sketch of the proof of Theorem~\ref{ThmWeakAssumption}.
  
\end{itemize}


\section{Interpolation, KL Bounds and Fokker-Planck Equation}
\label{sec:fokker}

Following~\cite{dalalyan2017theoretical}, the first step of the proof
is to construct a continuous-time interpolation for the discrete-time
algorithm~\eqref{eq:langevin-euler}.  In particular, we define a
stochastic process over the interval $t \in [\stepsize, (k+1) \stepsize]$ via
\begin{align}
\label{eq:def-interpolation}
\discretized{X}_t & \defn \discretized{X}_{k \stepsize} + \int_0^{t-k \stepsize} b(\discretized{X}_{k
  \stepsize}) ds + \int_{k \stepsize}^{t} d \discretized B_s. 
\end{align}
Let $\{ \discretized{\mathcal{F}}_t \, \mid \, t \geq 0\}$ be the
natural filtration associated with the Brownian motion $\{
\discretized{B}_t \, \mid \, t \geq 0\}$. Conditionally on
$\discretized{\mathcal{F}}_{k \stepsize}$, the process $\{
(\discretized{X}_t | \mathcal{F}_{k \stepsize}) \, \mid \, t \in [k
  \stepsize, (k+1) \stepsize ] \}$ is a Brownian motion with constant
drift $b(\discretized{X}_{k \stepsize})$ and starting at
$\discretized{X}_{k \stepsize}$. This interpolation has been used in
past work~\citep{dalalyan2017theoretical,cheng2017convergence}.  In
their work, the KL divergence between the law of processes $\{ X_t
\mid t \in [0,T] \}$ and $ \{ \discretized{X}_t \mid t \in [0, T] \}$
is controlled, via a use of the Girsanov theorem, by bounding
Radon-Nikodym derivatives.  This approach requires controlling the
quantity $\mathbb{E} \Vert b(\discretized{X}_t) - b(\discretized{X}_{k
  \stepsize}) \Vert_2^2$ for $t \in[k \stepsize,(k+1)\stepsize]$.  It
is should be noted that it scales as $O(\stepsize)$, due to the scale
of oscillation of Brownian motions.

In our approach, we overcome this difficulty by only considering the
KL divergence of the one-time marginal laws
$\kull{\mathcal{L}(\discretized{X}_T) }{ \mathcal{L}(X_T)}$.  Let us
denote the densities of $X_t$ and $\discretized{X}_t$ with respect to Lebesgue
measure in $\real^d$ by $\pi_t$ and $\discretized{\pi}_t$, respectively. It is
well-known that when $b$ is Lipschitz, then the density $\pi_t$
satisfies the Fokker-Planck equation
\begin{align}
  \label{eq:fokker_langevin}
  \frac{\partial \pi_t}{\partial t} = -  \nabla \cdot(\pi_t b) +
  \frac{1}{2} \Delta \pi_t,
\end{align}
where $\Delta$ denotes the Laplacian operator.  On the other hand, the
interpolated process $\discretized{X}_{k \stepsize}$ is not Markovian, and so does
not have a semigroup generator.  For this reason, it is difficult to
directly control the KL divergence between it and the original
Langevin diffusion. In the following lemma, we construct a different
partial differential equation that is satisfied by $\discretized{\pi}_t$.
\begin{lemma}
\label{lemma-Fokker-Planck-for-interploated}
The density $\discretized{\pi}_t$ of the process $\discretized{X}_t$ defined in
\eqref{eq:def-interpolation} satisfies the PDE
\begin{align}
  \label{eq:fokker_inter}
  \frac{\partial \discretized{\pi}_t}{\partial t}= - \nabla \cdot
  \left(\discretized{\pi}_t \discretized{b}_t \right)+\frac{1}{2}\Delta \discretized{\pi}_t
  \qquad \mbox{over the interval $t \in[k \stepsize,(k+1)\stepsize]$,}
\end{align}
where $\discretized{b}_t(x) \defn \mathbb{E}\left(b(\discretized{X}_{k \stepsize})\big|
\discretized{X}_t = x \right)$ is a time-varying drift term.
\end{lemma}

See Section~\ref{sec:appendix-fokker-planck} for the proof of this
lemma. The key observation is that, conditioned on the $\sigma$-field
$\discretized{\mathcal{F}}_{k \stepsize} = \sigma(\discretized{X}_t:0 \leq t \leq k \stepsize)$,
the process $\left \{ (\discretized{X}_t \mid\discretized{\mathcal{F}}_{k \stepsize}) \,
\mid \, t \in [\stepsize, (k+1) \stepsize] \right \}$ is a Brownian motion with
constant drift, whose conditional density $\discretized{\pi}_t \mid
\discretized{\mathcal{F}}_{k \stepsize}$ satisfies a Fokker-Planck equation. Taking
the expectation on both sides, and interchanging the integral with the
derivatives, we obtain the Fokker-Planck equation for the density
$\discretized{\pi}_t$ unconditionally.

In Lemma~\ref{lemma-Fokker-Planck-for-interploated}, we have a
Fokker-Planck equation with time-varying coefficients; it is satisfied
by the one-time marginal densities of the continuous-time
interpolation for \eqref{eq:langevin-euler}. This
representation provides convenient tool for bounding the time
derivative of KL divergence, a task to which we turn in the next
section.

\subsection{Proof of Lemma~\ref{lemma-Fokker-Planck-for-interploated}}
\label{sec:appendix-fokker-planck}

We first consider the conditional distribution of $(\discretized{X}_t:k \stepsize
\leq t \leq (k+1)\stepsize)$, conditioned on $\discretized{\mathcal{F}}_{k
  \stepsize}$. At time $t=k \stepsize$, it starts with an atomic mass (viewed as
Dirac $\delta$-function at point $\discretized{X}_{k \stepsize}$, which is a member
of the tempered distribution space $\mathcal{S}'$~\citep[see,
  e.g.,][]{rudin-functional}. Its derivatives and Hessian are
well-defined as well.)  For $t>k \stepsize$, this conditional density
follows the Fokker-Planck equation for a Brownian motion with constant
drift:
\begin{align}
\label{eq:fokker_bm}
\frac{\partial
  \left(\discretized{\pi}_t|_{\discretized{\mathcal{F}}_{k \stepsize}}\right)}{\partial
  t}= - \nabla \cdot \left(\discretized{\pi}_t|_{\discretized{\mathcal{F}}_{k \stepsize}}b(\discretized{X}_{k \stepsize})\right)+
\frac{1}{2}\Delta \discretized{\pi}_t|_{\discretized{\mathcal{F}}_{k \stepsize}},
\end{align}
where the partial derivatives are in terms of the dummy variable
$x$. Take expectations of both sides of \eqref{eq:fokker_bm}. By interchanging derivative and integration, we obtain the following identities.  Rigorous justification are provided below.
\begin{subequations}
\begin{align}
     \mathbb{E}\left(\frac{\partial
    \discretized{\pi}_t|_{\discretized{\mathcal{F}}_{k \stepsize}}}{\partial t}(x)\right) & =
  \frac{\partial \discretized{\pi}_t}{\partial t}(x)\label{subeq-id-time-derivative} \\ 
 	\mathbb{E} \left(\nabla
        \left(\discretized{\pi}_t|_{\discretized{\mathcal{F}}_{k \stepsize}}(x)b(\discretized{X}_{k \stepsize})
        \right) \right)
        & = \nabla \cdot
        \left(\discretized{\pi}_t(x)\mathbb{E}\left(b(\discretized{X}_{k \stepsize})\big|
        \discretized{X}_t=x \right)\right)\label{subeq-id-gradient-drift}\\
        \mathbb{E}\left(\Delta
\discretized{\pi}_t|_{\discretized{\mathcal{F}}_{k \stepsize}}\right) &= \Delta \discretized{\pi}_t.\label{subeq-id-laplacian}
\end{align}
\end{subequations}
\paragraph{Proof of equation~\eqref{subeq-id-time-derivative}:}
 We show:
\begin{align*}
  \mathbb{E}\left(\frac{\partial \discretized{\pi}_t
    \mid_{\discretized{\mathcal{F}}_{k \stepsize}}}{\partial
    t}(x)\right) =\int_{\real^d}\discretized{\pi}_{k
    \stepsize}(y)\frac{\partial \discretized{\pi}_t
    \mid_{\discretized{\mathcal{F}}_{k \stepsize}}}{\partial t}(x|y)
  dy \stackrel{(i)}{=}\frac{\partial}{\partial t}
  \int_{\real^d}\discretized{\pi}_{k \stepsize}(y) \discretized{\pi}_t
  \mid_{\discretized{\mathcal{F}}_{k \stepsize}}(x|y) dy =
  \frac{\partial \discretized{\pi}_t}{\partial t}(x),
  \end{align*}
  Applying Lemma~\ref{lemma-coarse-estimate} in Appendix~\ref{sect:appendix-coarse-estimates}, we can show that the
  density $\discretized{\pi}_{k \stepsize}$ has a tail decaying as
  $Ce^{- r\Vert y\Vert^2 }$. We then note that $\frac{\partial
    \discretized{\pi}_t|_{\discretized{\mathcal{F}}_{k
        \stepsize}}}{\partial t}(x|y)$ is equal to the semigroup
  generator of the conditional Brownian motion with constant drift,
  which also decays exponentially with $y$, in a small neighborhood of
  $t$, for fixed $x$. So the quantity $\discretized{\pi}_{k
    \stepsize}(y)\frac{\partial
    \discretized{\pi}_t|_{\discretized{\mathcal{F}}_{k
        \stepsize}}}{\partial t}(x|y)$ has a dominating function of
  the form of $C(1 + \Vert y\Vert) e^{-r \Vert y\Vert^2}$ in a small
  neighborhood of $t$.  Combining with the dominated convergence
  theorem justifies step (i).

\paragraph{Proof of equation~\eqref{subeq-id-gradient-drift}:} We have:
\begin{align*}
 	\mathbb{E} \left(\nabla
        \left(\discretized{\pi}_t|_{\discretized{\mathcal{F}}_{k
            \stepsize}}(x)b(\discretized{X}_{k \stepsize}) \right)
        \right)& = \int_{\real^d } \discretized{\pi}_{k
          \stepsize}(y) \nabla_x \cdot \left(
        \discretized{\pi}_t|_{\discretized{\mathcal{F}}_{k
            \stepsize}}(x|y) b(y) \right) dy\\ & \stackrel{(i)}{=}
        \nabla_x \cdot \int_{\real^d } \discretized{\pi}_{k
          \stepsize}(y)
        \discretized{\pi}_t|_{\discretized{\mathcal{F}}_{k
            \stepsize}}(x|y) b(y) dy\\ & \stackrel{(ii)}{=} \nabla
        \cdot
        \left(\discretized{\pi}_t(x)\mathbb{E}\left(b(\discretized{X}_{k
          \stepsize})\big| \discretized{X}_t=x \right)\right).
\end{align*}
In order to justify step (i), we first note that, according to
Assumption~\ref{assume-lipschitz-drift}, both of the functions $y
\mapsto b(y)$ and $y \mapsto \nabla_x \log
\discretized{\pi}_t|_{\discretized{\mathcal{F}}_{k \stepsize}}(x|y) $
grow at most linearly in $y$, for fixed $t$. By the rapid decay of the
tail of $\discretized{\pi}_t$ shown in
Lemma~\ref{lemma-coarse-estimate}, and the decay of the tail of
$\discretized{\pi}_t|_{\discretized{\mathcal{F}}_{k \stepsize}}(x|y)$
obtained by elementary results on the Gaussian density, we have a
dominating function of the form of $C ( 1 + \Vert y\Vert^2) e^{- r
  \Vert y\Vert^2}$. This justifies $(i)$ by the dominated convergence
theorem. Then $(ii)$ simply follows from the Bayes rule.
\paragraph{Proof of equation\eqref{subeq-id-laplacian}:} We similarly have:
\begin{align*}
    \mathbb{E}\left(\Delta
\discretized{\pi}_t|_{\discretized{\mathcal{F}}_{k \stepsize}}(x)\right) =  \int_{\real^d} \Delta_x \left(\discretized{\pi}_t|_{\discretized{\mathcal{F}}_{k \stepsize}}(x|y) \right) \discretized{\pi}_{k \stepsize}(y) dy.
\end{align*}
Note that $\Delta p (x) = (\Delta \log p + \Vert \nabla \log p\Vert^2 )p$ for any density function $p$. Since $\log \discretized{\pi}_t|_{\discretized{\mathcal{F}}_{k \stepsize}}(x|y)$ is a quadratic function in the variable $x$, its gradient is linear (it also grows at most linearly with $\Vert y\Vert$), and its Laplacian is constant. Therefore, we have a dominating function of form $C( 1 + \Vert y\Vert^2 )e^{- r\Vert y\Vert^2}$ for the integrand, which guarantees the interchange between the integral and the Laplacian operator. This leads to $\mathbb{E}\left(\Delta
\discretized{\pi}_t|_{\discretized{\mathcal{F}}_{k \stepsize}}(x)\right) = \Delta \discretized{\pi}_t (x)$.

Combining these identities yields
\begin{align*}
	\frac{\partial \discretized{\pi}_t}{\partial t}(x)=\nabla \cdot
        \left(\discretized{\pi}_t(x)\discretized{b}_t(x)\right)+\frac{1}{2}\Delta
        \discretized{\pi}_t,\quad t \in[k \stepsize,(k+1)\stepsize],
\end{align*}
where $\discretized{b}_t(x)=\mathbb{E}\left(b(\discretized{X}_{k \stepsize})\big|
\discretized{X}_t=x \right)$ for $t \in[k \stepsize,(k+1)\stepsize]$.

\section{Controlling the KL divergence: Proof of Proposition~\ref{prop-KL-derivative-fisher-moment}}
\label{sec:kl}

We now turn to the proof of
Proposition~\ref{prop-KL-derivative-fisher-moment}, which involves
bounding the derivative
\mbox{$\frac{d}{dt}\kull{\discretized{\pi}_t}{\pi_t)}$}. We first compute
the derivative using the Fokker-Planck equation established in
Lemma~\ref{lemma-Fokker-Planck-for-interploated}, and then upper bound
it by a regularity estimate of the density $\discretized{\pi}_{k \stepsize}$ and
moment bounds on $\discretized{X}_{k \stepsize}$. The key geometric intuition
underlying our argument is the following: if the drift $b$ is
second-order smooth and the initial distribution at each step is also
smooth, most of the Gaussian noise is cancelled out, and only
higher-order terms remain. This intuition is fleshed out in
Section~\ref{subsection-geometric-intuition}.

In the following lemma, we give an explicit upper bound on the KL
divergence between the one-time marginal distributions of the
interpolated process and the original diffusion, based on
Fokker-Planck equations derived above.
\begin{lemma}
\label{lemma-kl-time-derivative}
Suppose that the densties $\pi$ and $\discretized \pi$ satisfy the
Fokker-Planck equations~\eqref{eq:fokker_langevin}
and~\eqref{eq:fokker_inter}, respectively. Then
\begin{align}
\label{eq:boundkl}
\frac{d}{dt}\kull{\discretized{\pi}_t}{\pi_t} \leq
\frac{1}{2}\int_{\real^d}\discretized{\pi}_t(x)\Vert
\discretized{b}_t(x)-b(x)\Vert_2^2dx.
\end{align}
\end{lemma}
\noindent See Appendix~\ref{app:proof} for the proof of this claim.

It is worth noting the key difference between our approach and the
method of~\citet{dalalyan2017theoretical}, which is based on the
Girsanov theorem.  His analysis controls the KL divergence via the
quantity $\int_0^T \mathbb{E}\Vert b(\discretized{X}_{k
  \stepsize})-b(\discretized{X}_t)\Vert_2^2dt$, a term which scales as $O(\stepsize)$
even for the simple case of the Ornstein-Uhlenbeck process.  Indeed,
the Brownian motion contributes to an $O(\stepsize)$ oscillation in $\Vert
\discretized{X}_{k \stepsize}-\discretized{X}_t \Vert_2^2$, dominating other lower-order
terms.  By contrast, we control the KL divergence using the quantity
$\int_0^T \mathbb{E}\Vert
\discretized{b}_t(\discretized{X}_t)-b(\discretized{X}_t)\Vert_2^2dt$.  Observe that
$\discretized{b}_t$ is exactly the backward conditional expectation of
$b(\discretized{X}_{k \stepsize})$ conditioned on the value of $\discretized{X}_t$. Having
the conditional expectation inside (rather than outside) the norm
enables the lower-order oscillations to cancel out.

In the remainder of this section, we focus on bounding the integral on
the right-hand side of \eqref{eq:boundkl}. Since the
difference between $\discretized{X}_{k \stepsize}$ and $\discretized{X}_t$ comes mostly
from an isotropic noise, we may expect it to mostly cancel out. In
order to exploit this intuition, we use the third-order smoothness
condition (see Assumption~\ref{assume-smooth-drift}) so as to perform
the Taylor expansion
\begin{align}
\label{eq:taylor}
  \discretized{b}_t(x) - b(x) = \mathbb{E} \left(b(\discretized{X}_{k \stepsize}) -
  b(\discretized{X}_t) \big | \discretized{X}_t = x \right) =\nabla b(x) \mathbb{E}
  \left( \discretized{X}_{k \stepsize} - \discretized{X}_t \big | \discretized{X}_t=x \right) +
  \discretized{r}_t(x).
\end{align}
The reminder term $\discretized{r}_t(x)=\mathbb{E}\left(\int_0^1 s \nabla^2 b\left( (1-s)
\discretized{X}_t+s \discretized{X}_{k \stepsize} \right) \left[\discretized{X}_{k
    \stepsize}-\discretized{X}_t,\discretized{X}_{k \stepsize}-\discretized{X}_t \right] ds
\Big|\discretized{X}_t=x \right)$ is relatively easy to control, since it
contains a $\Vert \discretized{X}_{k \stepsize}-\discretized{X}_t \Vert_2^2$ factor, which
is already of order $O(\stepsize)$.  More formally, we have:
\begin{lemma}
\label{lemma-remainder}
Let us define $ \discretized{r}_t(x)\defn \mathbb{E}\left(\int_0^1
s \nabla ^2b((1\!-s)\discretized{X}_t\!+s \discretized{X}_{k \stepsize})[\discretized{X}_{k \stepsize}\!-\!\discretized{X}_t,\discretized{X}_{k \stepsize}\!-\!\discretized{X}_t]ds \Big|\discretized{X}_t\!=\!x \right)$. We have under Assumptions~\ref{assume-lipschitz-drift}
and~\ref{assume-smooth-drift}:
\begin{align*}
  \mathbb{E}\Vert \discretized{r}_t(\discretized{X}_t)\Vert_2^2 \leq
  8(t-k \stepsize)^4\hessianlip^2 \left(A_0^4+\smooth^4 \mathbb{E}\Vert
  \discretized{X}_{k \stepsize}\Vert_2^4 \right)+24(t-k \stepsize)^2\hessianlip^2 d^2, \text{
    for any } t \in[k \stepsize,(k+1)\stepsize).
\end{align*}
\end{lemma}
\noindent See Appendix~\ref{app:proof} for the proof of this claim.

It remains to control the first order term.  From
Assumption~\ref{assume-lipschitz-drift}, the Jacobian norm $\opnorm{
\nabla b(x)}$ is at most $\smooth$; accordingly, we only need to
control the norm of the vector $\mathbb{E}\left(\discretized{X}_{k \stepsize} -
\discretized{X}_t \big| \discretized{X}_t=x \right)$.  It corresponds to the
difference between the best prediction about the past of the path and
the current point, given the current information.  Herein lies the
main technical challenge in the proof of
Proposition~\ref{prop-KL-derivative-fisher-moment}, apart from the
construction of the Fokker-Planck equation for the interpolated
process. Before entering the technical details, let us first provide
some geometric intuition for the argument.


\subsection{Geometric Intuition}
\label{subsection-geometric-intuition}

Suppose that we were dealing with the conditional expectation of
$\discretized{X}_t$, conditioned on $\discretized{X}_{k \stepsize}$; in this case, the
Gaussian noise would completely cancel out (see \eqref{eq:def-interpolation}). However, we are indeed
reasoning backward, and $\discretized{X}_t$ itself is dependent with the
Gaussian noise $\int_{k \stepsize}^t d \discretized B_s$ added to this process.
It is unclear whether the cancellation occurs when
computing $\mathbb{E}\left(\discretized{X}_{k \stepsize}\big|\discretized{X}_t
\right)-\discretized{X}_t$. In fact, it occurs only under particular
situations, which turn out to typical for the
discretized process.
\begin{figure}[t]
  \begin{center}
    \begin{tabular}{ccc}
      \includegraphics[width=0.45 \linewidth]{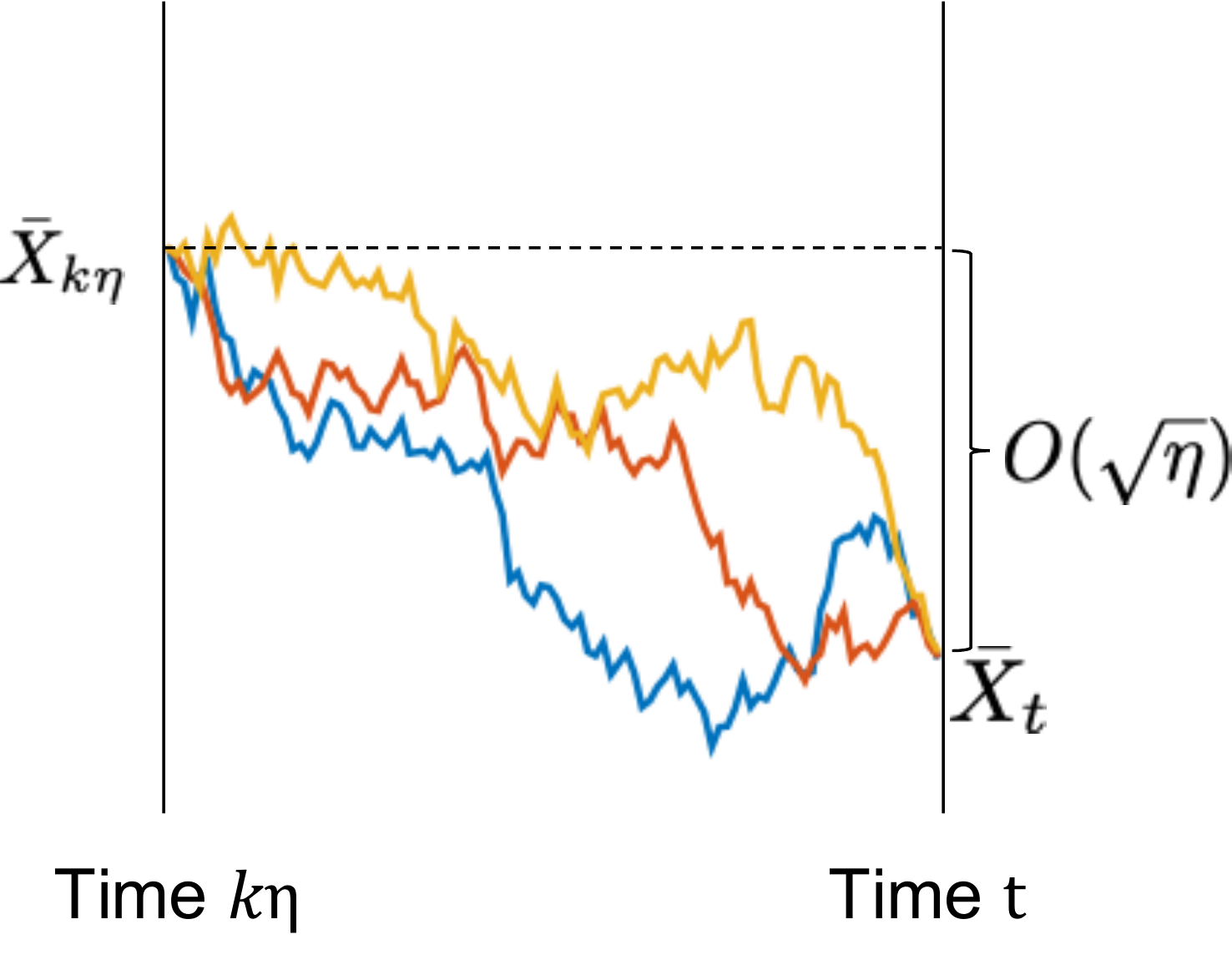} &&
      \includegraphics[width=0.45 \linewidth]{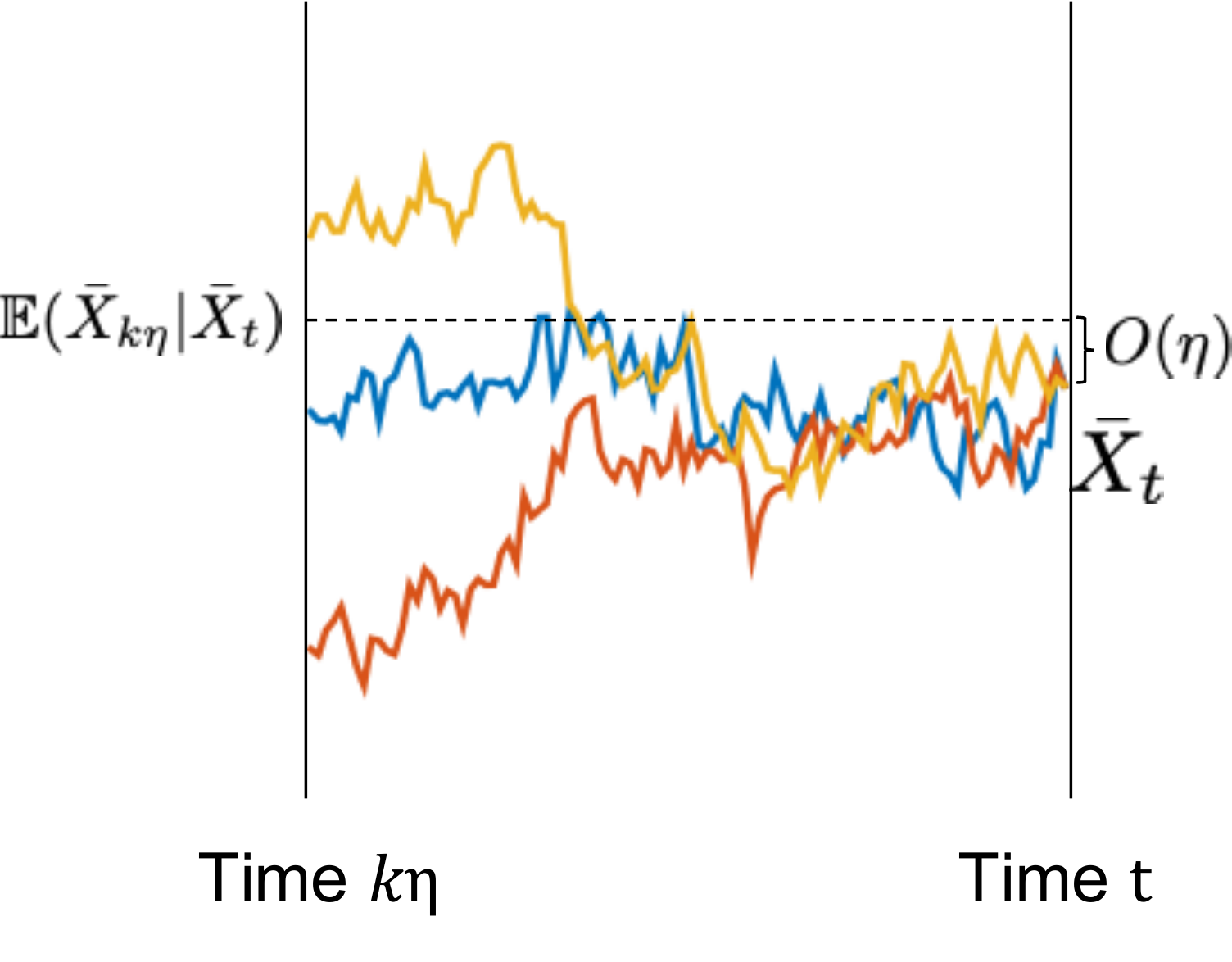} \\
      (a) && (b)
    \end{tabular}
    \caption{(a) Non-smooth initial distribution leads to
      $O(\sqrt{\stepsize})$ scaling.  (b) Smooth initial distribution:
      Oscillation cancels and leads to $O(\stepsize)$ error.}
\label{figure-bad-good}
  \end{center}
  \vspace{-.8cm}
\end{figure}

Due to the dependence between $\discretized{X}_t$ and Gaussian noise, we
cannot expect cancellation to occur in
general. Figure~\ref{figure-bad-good}(a) illustrates an extremal case,
where the initial distribution at time $k \stepsize$ is an atomic mass.
When we condition on the value at $\discretized{X}_{t}$ as well, the process
behaves like a Brownian bridge.  Consequently, it
makes no difference whether the conditional expectation is inside or
outside the norm: in either case, there is a term of the form $\Vert
\discretized{X}_{k \stepsize}-\discretized{X}_t \Vert_2$, which scales as $O(\sqrt{\stepsize})$.

On the other hand, as illustrated in Figure~\ref{figure-bad-good}(b),
if the initial distribution is uniform over some region, the initial
point is almost equally likely to be from anywhere around $\discretized{X}_t$, up
to the drift term, and most of the noise gets cancelled out. In
general, if the initial distribution is smooth, locally it looks
almost uniform, and similar phenomena should also hold true. Thus
we expect $\mathbb{E}(\discretized{X}_{k \stepsize}|\discretized{X}_t)-\discretized{X}_t$ to be
decomposed into terms coming from the drift and terms coming from the
smoothness of the initial distribution.


\subsection{Upper Bound via Integration by Parts}
\label{subsection-mse-to-fisher}

With this intuition in hand, we now turn to the proof itself.  In
order to leverage the smoothness of the initial distribution, we use
integration by parts to move the derivatives onto the density of
$\discretized{X}_{k \stepsize}$.  From Bayes' formula, we have
\begin{align}\label{eq:decis}
\mathbb{E}\left( \discretized{X}_{k \stepsize}\!-\!\discretized{X}_t \big| \discretized{X}_t\!=\!x \right)
\!=\! \int(y\!-\!x)p( \discretized{X}_{k \stepsize}\!=\!y|\discretized{X}_t\!=\!x)dy\! =\! \int (y\!-\!x)\textstyle
\frac{\discretized{\pi}_{k \stepsize}(y) p(\discretized{X}_t=x | \discretized{X}_{k
    \stepsize}=y)}{\discretized{\pi}_t(x)} dy.
\end{align}
Since the density $p(\discretized{X}_t = x \mid \discretized{X}_{k \stepsize} = y)$ is a
Gaussian centered at $y-(t-k \stepsize) b(y)$ with fixed covariance, the
gradient with respect to $y$ is the density itself times a linear
factor $x-y+(t-k \stepsize)b(y)$, with an additional factor depending on
the Jacobian of $b$. This elementary fact motivates a decomposition
whose goal is to express $\mathbb{E} \left( \discretized{X}_{k \stepsize} -
\discretized{X}_t \big | \discretized{X}_t = x \right)$ as the sum of the conditional
expectation of $\nabla \log \discretized{\pi}_{k \stepsize}$ and some other terms
which are easy to control.  More precisely, in order to expose a
gradient of the Gaussian density, we decompose the difference $y - x$
into three parts, namely $y-x = a_1(x,y) - a_2(x,y) - a_3(x,y)$, where
\begin{align*}
a_1(x,y) & \defn \left(I+(t-k \stepsize)\nabla b(y))(y-x+(t-k
\stepsize)b(y)\right), \\
a_2(x,y) & \defn (t-k \stepsize)\nabla b(y)(y-x+(t-k \stepsize)b(y)), \quad
\mbox{and} \\
a_3(x,y) & \defn (t-k \stepsize) b(y).
\end{align*}
We define the conditional expectations $I_i(x) \defn
\mathbb{E}\left(a_i(\discretized{X}_{k \stepsize},\discretized{X}_t)\big|\discretized{X}_{t}=x
\right)$ for $i=1,2,3$ and control the three terms separately.

Let us denote by $\varphi$ the $d$-dimensional standard Gaussian
density. The first term $I_1$ can directly be expressed in term of the gradient of $\varphi$:
\begin{align*}
    I_1(x) & = \int (I+(t-k \stepsize)\nabla b(y))(y-x+(t-k
    \stepsize)b(y))\varphi \left(\frac{x-y-(t-k \stepsize) b(y)}{\sqrt{t-k
        \stepsize}}\right)\frac{\discretized{\pi}_{k \stepsize}(y)}{\discretized{\pi}_t(x)}dy \\
& = (t-k \stepsize)\int \nabla_y \varphi \left(\frac{x-y-(t-k \stepsize)
      b(y)}{\sqrt{t-k \stepsize}}\right)\frac{\discretized{\pi}_{k
        \stepsize}(y)}{\discretized{\pi}_t(x)}dy,
\end{align*}
where we used the chain rule and   $\nabla \varphi(y)=-y \varphi(y)$. Thus, applying integration by parts, we write $I_1$ in a
revised form.
\begin{lemma}
  \label{lemma-i1-regularity}
For all $t \in[k \stepsize, (k+1)\stepsize]$, we have
\begin{align*}
I_1(x) = -( t - k \stepsize)\mathbb{E} \left(\nabla \log \discretized{\pi}_{k
  \stepsize}(\discretized{X}_{k \stepsize}) \big| \discretized{X}_t=x \right),
\end{align*}
and consequently,
\begin{align*}
    \mathbb{E}\Vert I_1(\discretized{X}_t)\Vert_2^2 \leq (t-k \stepsize)^2 \int
\discretized{\pi}_{k \stepsize}(x)\Vert \nabla \log \discretized{\pi}_{k
  \stepsize}(x)\Vert_2^2dx.
\end{align*}
\end{lemma}
\noindent See Section~\ref{AppProofI1} for the proof of this
lemma. 

It is clear from Lemma~\ref{lemma-i1-regularity} that a regularity
estimates on the moments of $\nabla \log \discretized{\pi}_{k \stepsize}(\discretized{X}_{k
  \stepsize})$ gives an $O(\stepsize^2)$ estimates on the squared
integral. Such a bound with reasonable dimension dependence is
nontrivial to obtain. This is postponed to section~\ref{sec:regu}.

The remaining two terms are relatively easy to control, as summarized
in the following:
\begin{lemma}
  \label{lemma-i2i3-bounds} 
Under Assumption~\ref{assume-lipschitz-drift}, the following bounds
hold for all $t \in[k \stepsize, (k+1)\stepsize ]$:
\begin{subequations}
  \begin{align}
    \label{EqnI2Bound}
\mathbb{E} \Vert I_2(\discretized{X}_t) \Vert_2^2 & \leq 3 (t-k \stepsize)^3 \smooth^2
d, \quad \mbox{and} \\
\label{EqnI3Bound}
\mathbb{E} \Vert I_3(\discretized{X}_t) \Vert_2^2 & \leq 2 \stepsize^2(A_0^2+ \smooth^2
\mathbb{E} \Vert \discretized{X}_{k \stepsize} \Vert_2^2).
\end{align}
\end{subequations}
\end{lemma}
\noindent See Section~\ref{AppProofI2I3} for the proof of this lemma. 

Combining the Taylor expansion~\eqref{eq:taylor} with the bounds from
Lemma~\ref{lemma-i1-regularity} and~\ref{lemma-i2i3-bounds} yields the
bound claimed in Proposition~\ref{prop-KL-derivative-fisher-moment}.

\subsection{Proof of Lemma~\ref{lemma-i1-regularity}}
\label{AppProofI1}

We prove here Lemma~\ref{lemma-i1-regularity} which controls the
dominant term $I_1$ of the decomposition of $\mathbb{E}\left(
\discretized{X}_{k \stepsize}\!-\!\discretized{X}_t \big|
\discretized{X}_t\!=\!x \right)$ in \eqref{eq:decis}. Recall $I_1$ is
expressed in term of the gradient of the Gaussian density:
\begin{align*}
    I_1(x) 
& = (t-k \stepsize)\int \nabla_y \varphi \left(\frac{x-y-(t-k \stepsize)
      b(y)}{\sqrt{t-k \stepsize}}\right)\frac{\discretized{\pi}_{k
        \stepsize}(y)}{\discretized{\pi}_t(x)}dy,
\end{align*}
where $\varphi$ is the $d$-dimensional standard Gaussian
density.
We first note the tail of the Gaussian density is
trivial, and the tail of $\discretized{\pi}_{k \stepsize}$ is justified by Appendix~\ref{sect:appendix-coarse-estimates}. Therefore we obtain applying integration by parts:   
\begin{align*}
  I_1(x) =&\int (I+(t-k \stepsize)\nabla b(y))(y-x+(t-k \stepsize)b(y))(2
  \pi(t-k \stepsize))^{-\frac{d}{2}}\\ &\quad \quad \quad \quad \quad
  \cdot \exp \left(-\frac{1}{2(t-k \stepsize)}\Vert x-y-(t-k \stepsize)
  b(y)\Vert_2^2 \right)\frac{\discretized{\pi}_{k
      \stepsize}(y)}{\discretized{\pi}_t(x)}dy \\ =&\int (t-k \stepsize)\nabla_y
  \exp \left(-\frac{1}{2(t-k \stepsize)}\Vert x-y-(t-k \stepsize)
  b(y)\Vert_2^2 \right)\frac{\discretized{\pi}_{k
            \stepsize}(y)}{\discretized{\pi}_t(x)}dy \\ =&-(t-k \stepsize)\int \exp
  \left(-\frac{1}{2(t-k \stepsize)}\Vert x-y-(t-k \stepsize) b(y)\Vert_2^2
  \right)\frac{\nabla_y \discretized{\pi}_{k \stepsize}(y)}{\discretized{\pi}_t(x)}dy
  \\ =&-(t-k \stepsize)\int \nabla_y \log \discretized{\pi}_{k
    \stepsize}(y)p(\discretized{X}_t=x|\discretized{X}_{k \stepsize}=y)\frac{\discretized{\pi}_{k
      \stepsize}(y)}{\discretized{\pi}_t(x)}dy \\ =&-(t-k
  \stepsize)\mathbb{E}\left(\nabla \log \discretized{\pi}_{k \stepsize}(\discretized{X}_{k
    \stepsize})\big|\discretized{X}_t=x \right).
\end{align*}
Then, applying the Cauchy-Schwartz inequality  yields
\begin{align*}
  \mathbb{E}\Vert I_1(\discretized{X}_t)\Vert_2^2=&(t-k \stepsize)^2
  \mathbb{E}\left \Vert \mathbb{E}\left(\nabla \log \discretized{\pi}_{k
    \stepsize}(\discretized{X}_{k \stepsize})\big|\discretized{X}_t \right)\right \Vert_2^2
  \\ \leq& (t-k \stepsize)^2 \mathbb{E}\Vert \nabla \log \discretized{\pi}_{k
    \stepsize}(\discretized{X}_{k \stepsize})\Vert_2^2=(t-k \stepsize)^2 \int \discretized{\pi}_{k
    \stepsize}\Vert \nabla \log \discretized{\pi}_{k \stepsize}\Vert_2^2.
\end{align*}
This last inequality concludes the proof of Lemma~\ref{lemma-i1-regularity}.

\subsection{Proof of Lemma~\ref{lemma-i2i3-bounds}}
\label{AppProofI2I3}

Recall that this lemma provides bounds on the remaining two terms
$I_2(x)$ and $I_3(x)$ of the decomposition of $\mathbb{E}\left(
\discretized{X}_{k \stepsize}\!-\!\discretized{X}_t \big|
\discretized{X}_t\!=\!x \right)$ in \eqref{eq:decis}. We split our
proof into two parts, corresponding to the two bounds.

\paragraph{Proof of the bound~\eqref{EqnI2Bound}:}

We directly bound the Jacobian matrix using
Assumption~\ref{assume-lipschitz-drift}.
\begin{align*}
\frac{\Vert I_2(x)\Vert_2}{t-k\stepsize} & = \left \Vert \int \nabla
b(y)(y-x+(t-k \stepsize)b(y))(2 \pi(t-k \stepsize))^{-\frac{d}{2}}\exp
\left(-\frac{\Vert x-y-(t-k \stepsize) b(y)\Vert_2^2}{2 (t-k
  \stepsize)} \right)\frac{\discretized{\pi}_{k
    \stepsize}(y)}{\discretized{\pi}_t(x)}dy \right \Vert_2 \\
& \leq \smooth \int \Vert (y - x + (t - k \stepsize) b(y) \Vert_2
\frac{\discretized{\pi}_{k \stepsize}(y)}{\discretized{\pi}_t(x)} p(
\discretized{X}_t = x \big | \discretized{X}_{k \stepsize} = y ) dy \\
& = \smooth \mathbb{E} \left( \Vert \discretized{X}_{k \stepsize} +
(t-k \stepsize) b(\discretized{X}_{k \stepsize}) - \discretized{X}_t
\Vert_2 \Big| \discretized{X}_t = x \right)\\
& = \smooth \mathbb{E}\left(\Vert \int_{k \stepsize}^t dB_s \Vert_2
\Big|\discretized{X}_t=x \right).
\end{align*}
Plugging into the squared integral yields
\begin{align*}
\mathbb{E} \Vert I_2(\discretized{X}_t)\Vert_2^2 \leq (t-k
\stepsize)^2 \smooth^2 \mathbb{E} \left(\mathbb{E}\left(\Vert \int_{k
  \stepsize}^t dB_s \Vert \Big | \discretized{X}_t \right)\right)^2
\leq (t-k \stepsize)^2\smooth^2 \mathbb{E} \Vert \int_{k \stepsize}^t
dB_s \Vert_2^2 \leq 3 (t-k \stepsize)^3 \smooth^2d.
\end{align*}

\paragraph{Proof of the bound~\eqref{EqnI3Bound}:}

The size of norm of $I_3$ is determined largely by
$b(\discretized{X}_{k \stepsize})$, which can be controlled using
Assumption~\ref{assume-lipschitz-drift}:
\begin{align*}
 \mathbb{E} \Vert I_3 (\discretized{X}_t) \Vert_2^2 = (t-k
 \stepsize)^2 \mathbb{E}\Vert \mathbb{E}(b(\discretized{X}_{k
   \stepsize}) | \discretized{X}_t)\Vert_2^2 \leq \stepsize^2
 \mathbb{E}\Vert b(\discretized{X}_{k \stepsize})\Vert_2^2 \leq 2
 \stepsize^2(A_0^2+ \smooth^2 \mathbb{E}\Vert \discretized{X}_{k
   \stepsize}\Vert_2^2).
\end{align*}


\section{Regularity and Moment Estimates}
\label{sec:regu}

From the previous section, we have upper bounded the time derivative
of the KL divergence between the Langevin diffusion and its Euler
discretization, using the Fisher information of $\pi_{k \stepsize}$ and the
moment of $\discretized{X}_{k \stepsize}$. In order to show that the above estimate
is $O(\stepsize^2)$, we derive, in the next section, upper bounds on the
Fisher information and the moments which are independent of the step
size.

Bounding the discretization error essentially relies on a $L^2$
estimate of $\nabla \log \discretized{\pi}_{k \stepsize}$, and a higher order
moment of $\discretized{X}_{k \stepsize}$. In this section, we provide
non-asymptotic bounds for both quantities. The regularity estimate is
using a variant of the famous De Bruijn identity that relates Fisher
information to entropy. This stands in sharp contrast to classical PDE
regularity theory, which suffers from exponential dimension
dependencies. The moment estimate comes from a standard martingale
argument, but with explicit dependence on all the parameters.

\subsection{Proof of Proposition~\ref{prop-fisher-grid}}

We now turn to the proof of Proposition~\ref{prop-fisher-grid}, which
gives a control on the Fisher information term needed by
Proposition~\ref{prop-KL-derivative-fisher-moment}.  We first bound
the time integral of $\int \discretized \pi_t \Vert \log \discretized \pi_t
\Vert_2^2$, and then relate it to the average at the grid points. The
techniques introduced are novel and of independent interests.

The De Bruijn identity relates the time derivative of the KL
divergence with the Fisher information for the heat
kernel~\citep{Cover}. We establish an analogous result for the
Fokker-Planck equation constructed in
Lemma~\ref{lemma-Fokker-Planck-for-interploated}. This serves as a
starting point of the regularity estimate used in this paper, though
going from time integral to discrete grid points still takes effort.
\begin{lemma}\label{lemma-regularity-time-integral}
	For the time-marginal densities $\discretized{\pi}_t$ of the interpolated process $\discretized{X}_t$, we have:
	\begin{align*}
	\int_0^T \int \discretized{\pi}_t(x)\Vert \nabla \log \discretized{\pi}_t(x)\Vert_2^2dxdt \leq 4 \left(h_0+\sigma_0^{-2} \mathbb{E}\Vert \discretized{X}_T \Vert_2^2 + H (\pi_0)\right)+16 \int_0^T \left(A_0^2+\smooth^2 \mathbb{E}\Vert \discretized{X}_t \Vert_2^2 \right)dt,
\end{align*}
where $H (\cdot)$ denotes the differential entropy.
\end{lemma}
See Section~\ref{sec-appendix-regularity-time-integral} for the
proof. 


Lemma~\ref{lemma-regularity-time-integral} gives control on the
average of the second order regularity estimate. However, we want
bound for this quantity evaluated at the grid points $\{k \stepsize
\}_{k=1}^{+\infty}$. To relate back to grid points, we use to
discrete-time arguments, by splitting the transformation from
$\discretized{\pi}_{t}$ to $\discretized{\pi}_{k \stepsize}$ into two
parts, and mimic the forward Euler algorithm. The following lemma
gives control on the relative difference between the integral at time
$k \stepsize$ and $t \in[(k-1)\stepsize,k \stepsize]$. The proof is
postponed to
Section~\ref{sect:proof-regularity-relative-difference}.
\begin{lemma}
  \label{lemma-regularity-relative-difference}
For any $\stepsize \in \big(0, \frac{1}{2\smooth} \big)$ and $t_0 \in
[(k-1)\stepsize,k \stepsize]$, we have:
  \begin{align}
    \label{eq:regurelat}
    \int \discretized{\pi}_{k \stepsize} \Vert \nabla \log
    \discretized{\pi}_{k \stepsize}\Vert_2^2 \leq 8 \int \discretized{\pi}_{t_0} \Vert \nabla
    \log \discretized{\pi}_{t_0} \Vert_2^2 + 32 \stepsize^2 d^2 \hessianlip^2.
  \end{align}
\end{lemma}
Taking averages over $t_0 \!\in\![(k\!-\!1)\stepsize,k \stepsize]$ in
equation~\eqref{eq:regurelat} and then summing over $k$ completes the
proof of Proposition~\ref{prop-fisher-grid}.


\subsubsection{Proof of Lemma~\ref{lemma-regularity-time-integral}}
\label{sec-appendix-regularity-time-integral}

Our general strategy is to relate the Laplacian operator in the
semigroup generator, with the one that naturally comes from applying
integration by parts to the Fisher information. Note that in the
second step we use the divergence theorem, which is justified by
Remark~\ref{remark-integration-by-parts} in
Appendix~\ref{sect:appendix-coarse-estimates}.
\begin{align}
\label{eq:intpart}
\int \discretized{\pi}\Vert \nabla \log \discretized{\pi}\Vert_2^2 =
\int \langle \nabla \discretized{\pi},\nabla \log
\discretized{\pi}\rangle = -\int \log \discretized{\pi} \Delta
\discretized{\pi}.
\end{align}
On the other hand, the semigroup generator for the time-inhomogeneous
process is given by
\begin{align}
\label{eq:semi}
\discretized{\mathcal{L}}_t \discretized{\pi} = - \nabla \cdot
\left(\discretized{\pi}\discretized{b}\right)+\frac{1}{2}\Delta
\discretized{\pi}.
\end{align}
Putting together equations~\eqref{eq:intpart} and~\eqref{eq:semi}
yields
\begin{align*}
  \int \discretized{\pi}\Vert \nabla \log \discretized{\pi}\Vert_2^2 &
  = -2 \int (\discretized{\mathcal{L}}_t \discretized{\pi})\log
  \discretized{\pi} + 2 \int \nabla \cdot
  \left(\discretized{\pi}\discretized{b}\right)\log \discretized{\pi}
  \\
& = -2 \int \frac{\partial \discretized{\pi}}{\partial t}\log
  \discretized{\pi}+2 \int \nabla \cdot
  \left(\discretized{\pi}\discretized{b}\right)\log
  \discretized{\pi}\\ =&-2 \frac{d}{dt}\left(\int
  \discretized{\pi}\log \frac{\discretized{\pi}}{\pi_0}\right)-2
  \frac{d}{dt} \mathbb{E}_{\discretized{\pi}}\log \pi_0-2 \int
  \discretized{\pi}\left(\nabla \log \discretized{\pi}\cdot
  \discretized{b}\right) \\
  & \leq -2 \frac{d}{dt}\left(\int \discretized{\pi}\log
  \frac{\discretized{\pi}}{\pi_0}\right)-2
  \frac{d}{dt}\mathbb{E}_{\discretized{\pi}}\log \pi_0 + 4 \int
  \discretized{\pi}\Vert \discretized{b} \Vert_2^2 + \frac{1}{2} \int
  \discretized{\pi} \Vert \nabla \log \discretized{\pi} \Vert_2^2.
\end{align*}
This directly yields to:
\begin{align}
\label{eq:tempo}
\int \discretized{\pi} \Vert \nabla \log \discretized{\pi} \Vert_2^2 \leq -4
\frac{d}{dt} \left(\int \discretized{\pi} \log
\frac{\discretized{\pi}}{\pi_0}\right)-4 \frac{d}{dt} \mathbb{E}_{\discretized{\pi}}\log \pi_0 + 8
\int \discretized{\pi}\Vert \discretized{b}\Vert_2^2.
\end{align}
Under Assumption~\ref{assume-smooth-initialize}, we have
$-\mathbb{E}_{\discretized{\pi}}\log \pi_0 \leq
h_0+\sigma_0^{-2}\mathbb{E}\Vert \discretized{X}_t \Vert_2^2$. On the other
hand, $\int \discretized{\pi}\log \frac{\discretized{\pi}}{\pi_0}=\kull{\discretized{\pi}}{\pi_0} \geq 0$. Hence, plugging into \eqref{eq:tempo} and integrating
we obtain the desired result:
\begin{align*}
  \int_0^T \int \discretized{\pi}\Vert \nabla \log \discretized{\pi}\Vert_2^2dt \leq 4
  \left(h_0+\sigma_0^{-2} \mathbb{E}\Vert \discretized{X}_T
  \Vert_2^2 + H (\pi_0)\right)+16 \int_0^T \left(A_0^2 + \smooth^2 \mathbb{E}\Vert
  \discretized{X}_t \Vert_2^2 \right) dt.
\end{align*}

\subsubsection{Proof of Lemma~\ref{lemma-regularity-relative-difference}}\label{sect:proof-regularity-relative-difference}
The proof involves a sequence of auxiliary lemmas. We first show that
the transition from $\discretized{\pi}_{t}$ to $\discretized{\pi}_{k \stepsize}$ can be
viewed as a discrete-time update. 
\begin{lemma}
 \label{lemma-mimic-forward-euler}
 For a given $t_0 \in [(k-1)\stepsize,k \stepsize]$, define the random
 variable
 \begin{align} \discretized{Y}_{k \stepsize} & \defn \discretized{X}_{t_0}+(k
 \stepsize-t_0)\discretized{b}_{t_0}(\discretized{X}_{t_0})+(k
 \stepsize-t_0)^{\frac{1}{2}}\mathcal{N}(0,I).
 \end{align}
 Then we have $\discretized{Y}_{k \stepsize} \stackrel{d}{=} \discretized{X}_{k \stepsize}$.
\end{lemma}

Using Lemma~\ref{lemma-mimic-forward-euler}, we can see $\discretized{\pi}_{k
  \stepsize}$ as the consequence of a nonlinear transform and heat kernel
performed on $\discretized{\pi}_{t_0}$. For the first part, we can directly
bound it, as long as $\stepsize$ is not too large:
\begin{lemma}\label{lemma-fisher-one-step-blow-up}
Let $\phi: \real^d \rightarrow \real^d, \phi(x)=x+(t-k
\stepsize_0)\discretized{b}_{t_0}(x)$. For $\stepsize<\frac{1}{2\smooth}$, Let
$Z=\discretized{X}_{t_0}+(k \stepsize-t_0)\discretized{b}_{t_0}(\discretized{X}_{t_0})$, and let
$p(\cdot)$ be the density of $Z$. We have:
 \begin{align*}
   \int p(z)\Vert \nabla_z \log p(z)\Vert_2^2dz \leq 8 \int
   \discretized{\pi}_{t_0}(x)\Vert \nabla_x \log
   \discretized{\pi}_{t_0}(x)\Vert_2^2dx+32 \stepsize^2d^2\hessianlip^2.
 \end{align*}
\end{lemma}

The second term is harmless because it is of order $O(\stepsize^2)$,
leading to $O(\stepsize^4)$ in the final bound, and the first term blows up
the regularity estimate by factor 8.

In our next step, we are going to relate the $L^2$ regularity integral
of $p$ to that of $\discretized{\pi}_{k \stepsize}$, and therefore finish
establishing the connection between the integral at grid points and at
arbitrary time point.

First of all, we note the fact that the transition from $p$ to
$\discretized{\pi}_{k \stepsize}'$ follows a heat equation in
$\real^d$. Concretely, consider the equation:
\begin{align*}
    \frac{\partial u_s}{\partial s}(x)=\Delta u_s(x),\quad
    u_{t_0}(x)=p(x) \qquad \mbox{for all $x \in \real^d$},
\end{align*}
with $s \in[t_0,k \stepsize]$, the unique solution satisfies $u_{k
  \stepsize}=\discretized{\pi}_{k \stepsize}'=\discretized{\pi}_{k \stepsize}$ according to
Lemma~\ref{lemma-mimic-forward-euler}. A nice property about Fisher
information is that, it is non-increasing along the flow of heat
kernel:
\begin{lemma}\label{lemma-non-increasing-fisher-heat}
For the heat equation $\frac{\partial u_t}{\partial t}=\Delta u$ with
$u_0 \geq 0$, $\int u_0(x)dx=1$, and $u_0$ satisfying the conditions
in Appendix~\ref{sect:appendix-coarse-estimates}, we have that $\int_{\real^d} u_t(x) \Vert \log u_t(x)
\Vert dx$ is non-increasing in $t$.
\end{lemma}

Since $\stepsize \in (0, \frac{1}{2\smooth}]$ by assumption, for $t_0
\in[(k-1)\stepsize,k \stepsize]$, we obtain:
\begin{align}
    \int \discretized{\pi}_{k \stepsize}\Vert \nabla \log \discretized{\pi}_{k
      \stepsize}\Vert_2^2 \leq \int p \Vert \nabla \log p \Vert_2^2 \leq 8
    \int \discretized{\pi}_{t_0}\Vert \nabla \log \discretized{\pi}_{t_0}\Vert_2^2+32
    \stepsize^2d^2\hessianlip^2,
\end{align}
which completes the proof of
Lemma~\ref{lemma-regularity-relative-difference}.


\subsection{Moment Estimate under Dissipative Assumption}
\label{sec:momentestimate}

In this section, we bound the moments of the process $\discretized X_{k \stepsize}$
along the path of the discretized Langevin diffusion.  In order to do
so, we leverage Assumption~\ref{assume-strong-dissipative}, as stated
in the following:
\begin{lemma}
  \label{lemma-tail-strong-dissipative}
  Suppose that Assumptions~\ref{assume-strong-dissipative}
  and~\ref{assume-smooth-initialize} hold for the interpolated
  process~\eqref{eq:def-interpolation}. Then there is a universal
  constant $C > 0$ such that
\begin{align} \sup_{t \geq 0} \left(
  \mathbb{E}\Vert \discretized{X}_t \Vert_2^p \right)^{\frac{1}{p}}\leq
  C\left( \sigma_0 \sqrt{pd} + \sqrt{\frac{p + \beta + d}{\mu}}
  \right) \qquad \mbox{for all $p \geq 1$.}
  \end{align}
\end{lemma}
The proof of this lemma is based on martingale $L^p$
estimates and the Burkholder-Davis-Gundy
inequality~\citep{MR0400380}. The details are postponed to
Appendix~\ref{sec:prooflemma-tail-strong-dissipative}. It is worth noting the bound depends
polynomially on the parameters $(\mu, \beta)$ in
Assumption~\ref{assume-strong-dissipative}.

Without Assumption~\ref{assume-strong-dissipative} and control on the
directions of the drift at a far distance, the moment of the iterates
can exponentially blow up. A simple counterexample is to let the
potential function be $U(x)=-\Vert x \Vert_2^2$ and $b(x)=x$. Then it
is easy to see that $\Vert \discretized{X}_t \Vert_2 \gtrsim e^{T}$ in this
setup. 
this exponential growth,
However Assumption~\ref{assume-strong-dissipative} can actually be significantly
weakened---as long as the potential function is non-negative. This comes at the cost of a
worse dependence (still polynomial) on $T$.

\subsection{Moment Estimates without Dissipative Assumptions}

Note that Lemma~\ref{lemma-tail-strong-dissipative} requires the
distant dissipative assumption~\ref{assume-strong-dissipative}. This
assumption can be relaxed with a slightly worse dependence on $T$, as
long as the potential function is non-negative. In this section, we
assume $b = - \nabla f$ with $f (x) \geq 0$ for any $x \in \real^d$.
Under these conditions, we have the following:
\begin{lemma}\label{lemma:moment-24-only-non-neg}
Suppose Assumption~\ref{assume-smooth-initialize}
and~\ref{assume-lipschitz-drift} holds, for the
process~\eqref{eq:def-interpolation} with $b = - \nabla f$ and $f \geq
0$, there is a universal constant $C > 0$, such that:
\begin{align*}
  \sup_{0 \leq k \leq T / \stepsize} \left( \Exs \vecnorm{
    \discretized{X}_{k \stepsize } }{2}^{4} \right) \leq C \cdot
  \left( f(0)^2 + \smooth^2 T^2 \sigma_0^4 (h_0 + d)^2 + \smooth^2
  T^4 d^2 \right),
\end{align*}
for some universal constant $C > 0$.
\end{lemma}
By plugging the fourth moment obtained by
Lemma~\ref{lemma:moment-24-only-non-neg} into Propostion~\ref{prop-KL-derivative-fisher-moment} and Proposition~\ref{prop-fisher-grid}, Theorem~\ref{ThmWeakAssumption}
can then be established.

\subsubsection{Proof of Lemma~\ref{lemma:moment-24-only-non-neg}}
In this section, we present the proof of the moment bound given in
Lemmas~\ref{lemma:moment-24-only-non-neg}.

Let the process $\interpolated{X}_t$ be the time-inhomogeneous diffusion process defined by generator $- \inprod{\discretized{b}_t}{\nabla} + \frac{1}{2}\Delta$, starting from $\discretized{X}_0$. By Lemma~\ref{lemma-Fokker-Planck-for-interploated}, the one-time marginal laws of $\interpolated{X}_t$ and $\discretized{X}_t$ are the same. So we only need to show the moment bounds for $\interpolated{X}_t$.

    By It\^{o}'s formula, we have:
    \begin{align*}
        \vecnorm{\interpolated{X}_t}{2}^2 = \vecnorm{\interpolated{X}_{0}}{2}^2  - 2\int_0^t \inprod{\interpolated{X}_s}{\discretized{b}_s (\interpolated{X}_s)} ds + 2\int_0^t \inprod{\interpolated{X}_s}{d B_s} + t d.
    \end{align*}
    Taking expectation for both sides, we obtain:
    \begin{align*}
        \Exs \vecnorm{\interpolated{X}_t}{2}^2 \leq & \Exs \vecnorm{\interpolated{X}_0}{2}^2 + 2 \int_0^t \Exs \left( \vecnorm{\interpolated{X}_s}{2} \cdot \vecnorm{ \discretized{b}_s (\interpolated{X}_s)}{2} \right) ds + td \\
        \leq & \Exs \vecnorm{\interpolated{X}_0}{2}^2 + 2 \left(\int_0^t \Exs \vecnorm{
        \interpolated{X}_s}{2}^2 ds \right)^{\frac{1}{2}} \left(\int_0^t \Exs \vecnorm{
        \discretized{b}_s (\interpolated{X}_s) }{2}^2 ds \right)^{\frac{1}{2}} + td\\
        \leq & \Exs \vecnorm{\interpolated{X}_0}{2}^2 + 2 \left(\int_0^t \Exs \vecnorm{
        \interpolated{X}_s}{2}^2 ds \right)^{\frac{1}{2}} \left( \stepsize \sum_{k = 0}^{t / \stepsize} \Exs \vecnorm{
        \nabla f (\interpolated{X}_{k \stepsize}) }{2}^2 \right)^{\frac{1}{2}} + td.
    \end{align*}
    If an upper bound on average mean squared norm of $\nabla f(\discretized{X}_{k \stepsize})$ can be obtained, the conclusion directly follows from solving an ordinary differential inequalities using variants of Gr\"{o}nwall lemma. Therefore, we need the following lemma about the squared gradient norms:
        \begin{lemma} \label{lemma:gradient-sum-bound-24-nonnegative}
        Suppose Assumption~\ref{assume-lipschitz-drift} and Assumption~\ref{assume-smooth-initialize} holds, for the process~\eqref{eq:def-interpolation} with $b = - \nabla f, ~f \geq 0$ and $\stepsize < 1 / 2 \smooth$, there is a universal constant $C > 0$, such that:
    \begin{align*}
        \Exs \stepsize \sum_{k = 0}^{T / \stepsize} \vecnorm{ \nabla f(\discretized{X}_{k \stepsize } ) }{2}^2 \leq & 2 \Exs f(\discretized{X}_0) + \smooth T d,\\
        \Exs \left(\stepsize \sum_{k = 0}^{T / \stepsize} \vecnorm{ \nabla f(\discretized{X}_{k \stepsize } ) }{2}^2 \right)^2 \leq &  12 \Exs f(\discretized{X}_0)^2 + 24 \Exs f(\discretized{X}_0) + 12 T \smooth d + 9 \smooth^2 T^2 d^2.
    \end{align*}
    \end{lemma}
    The proof of this lemma is postponed to Appendix~\ref{sec:appendix-proof-lemma-24-grad-norm}. The main idea of the proof is straightforward: large norm of $\nabla f$ will force the value of $f$ to go down along the dynamics of Langevin algorithm. Since $f$ is non-negative, the average mean squared norm can be bounded using the initial value of $f$. However, the Gaussian noise is non-trivial to deal with. The combinatorial techniques used in the proof of Lemma~\ref{lemma:gradient-sum-bound-24-nonnegative} are only able to deal with up to fourth moment. Fortunately, this is what the proof of Theorem~\ref{ThmWeakAssumption} needs.
    
    Plugging the first bound in Lemma~\ref{lemma:moment-24-only-non-neg} into above upper bound for $\Exs \vecnorm{\interpolated{X}_t}{2}^2$, and taking supremum with $t \in [0, T]$ for both sides, we obtain:
    \begin{align*}
        \sup_{0 \leq t \leq T}  \Exs \vecnorm{\interpolated{X}_t}{2}^2 \leq & \Exs \vecnorm{\interpolated{X}_0}{2}^2 + 2 \sqrt{ ( 2 \Exs f(\interpolated{X}_0) + \smooth T d) \int_0^T \Exs \vecnorm{
        \interpolated{X}_s}{2}^2 ds } + Td\\
        \leq & \Exs \vecnorm{\interpolated{X}_0}{2}^2 + 2 \sqrt{ ( 2 \Exs f(\interpolated{X}_0) + \smooth T d) T \sup_{0 \leq t \leq T} \Exs \vecnorm{
        \interpolated{X}_s}{2}^2 ds } + Td.
    \end{align*}
    Solving the quadratic equation, we obtain:
    \begin{align*}
         \sup_{0 \leq t \leq T}  \Exs \vecnorm{\interpolated{X}_t}{2}^2 \leq 2 \Exs \vecnorm{\interpolated{X}_0}{2}^2 + 4 T (\Exs f(\interpolated{X}_0) + \smooth T d ) + 2 Td.
    \end{align*}
    
    For the fourth moment, using Young's inequality, we obtain:
    \begin{align*}
        \Exs \vecnorm{\interpolated{X}_t}{2}^4 \leq 4 \Exs \vecnorm{\interpolated{X}_0}{2}^4 + 16 \underbrace{\Exs \left( \int_0^t \inprod{\interpolated{X}_s}{\discretized{b}_s (\interpolated{X}_s)} ds \right)^2}_{T_1} + 16 \underbrace{\Exs \left( \int_0^t \inprod{\interpolated{X}_s}{d B_s} \right)^2}_{T_2} + 4 t^2 d^2.
    \end{align*}
    For $T_1$, applying the Cauchy-Schwartz inequality multiple times, we have:
    \begin{align*}
        T_1 \leq & \Exs \left( \int_0^t \vecnorm{\interpolated{X}_s}{2} \cdot \vecnorm{\discretized{b}_s (\interpolated{X}_s)}{2} ds \right)^2\\
        \leq& \Exs \left( \int_0^t \vecnorm{\interpolated{X}_s}{2}^2 ds \cdot \int_0^t \vecnorm{\discretized{b}_s (\interpolated{X}_s) }{2}^2 ds \right) \\
        \leq & \left( \Exs \left( \int_0^t \vecnorm{\interpolated{X}_s}{2}^2 ds \right)^2 \right)^{\frac{1}{2}} \cdot \left( \Exs \left( \int_0^t \vecnorm{\discretized{b}_s (\interpolated{X}_s) }{2}^2 ds \right)^2 \right)^{\frac{1}{2}} \\
        \leq & \left( t \int_0^t \Exs \vecnorm{\interpolated{X}_s}{2}^4 ds  \right)^{\frac{1}{2}} \left( \Exs \left( \stepsize \sum_{k = 0}^{t / \stepsize} \vecnorm{\nabla f (X_{k \stepsize})}{2}^2 \right)^2 \right)^{\frac{1}{2}}\\
        \leq & 12 (\sqrt{\Exs f(\interpolated{X}_0)^2} + \smooth T d)\left( t \int_0^t \Exs \vecnorm{\interpolated{X}_s}{2}^4 ds  \right)^{\frac{1}{2}}
    \end{align*}
 For $T_2$, combining the It\^{o} isometry with above bounds on the
 expected norm yields
    \begin{align*}
        T_2 = \int_0^t \Exs \vecnorm{\interpolated{X}_s}{2}^2 ds \leq
        2 t \Exs \vecnorm{\interpolated{X}_0}{2}^2 + 4 t (t \Exs
        f(\interpolated{X}_0) + t^2 \smooth d ) + 2 t^2 d.
    \end{align*}
Similar to the second moment, by taking supremum over time, we obtain
 \begin{multline*}
   \sup_{0 \leq t \leq T} \Exs \vecnorm{\interpolated{X}_t}{2}^4 \leq
   4 \Exs \vecnorm{\interpolated{X}_0}{2}^4 + 192 \left( \sqrt{\Exs
     f(\interpolated{X}_0)^2} + \smooth T d \right) \left( T \int_0^T
   \Exs \vecnorm{\interpolated{X}_s}{2}^4 ds \right)^{\frac{1}{2}} \\
 + 16 \left( 2 T \Exs \vecnorm{\interpolated{X}_0}{2}^2 + 4 T (T \Exs
 f(\interpolated{X}_0) + T^2 \smooth d ) + 2 T^2 d \right) + 4 T^2
 d^2 \\
 \leq C \cdot \Exs \vecnorm{\interpolated{X}_0}{2}^4 + C T \left(
 \sqrt{\Exs f(\interpolated{X}_0)^2} + \smooth T d \right) \sqrt{
   \sup_{0 \leq t \leq T} \Exs \vecnorm{\interpolated{X}_s}{2}^4} \\ +
 C \left( T \Exs \vecnorm{\interpolated{X}_0}{2}^2 + T^2 \Exs
 f(\interpolated{X}_0) + T^3 \smooth d + T^2 d^2 \right).
 \end{multline*}
Solving the quadratic equation yields
\begin{align*}
  \sup_{0 \leq t \leq T} \Exs \vecnorm{\interpolated{X}_s}{2}^4 & \leq
  C \cdot \Exs \vecnorm{\interpolated{X}_0}{2}^4 + C T^2 \left( \Exs
  f(\interpolated{X}_0)^2 + \smooth^2 T^2 d^2 \right) + C T \Exs
  \vecnorm{\interpolated{X}_0}{2}^2\\
& \leq C' \cdot \left( f(0)^2 + \smooth^2 T^2 \Exs
  \vecnorm{\interpolated{X}_0}{2}^4 + \smooth^2 T^4 d^2 \right) \\
& \leq C' \cdot \left( f(0)^2 + \smooth^2 T^2 \sigma_0^2 (h_0 + d)^2 +
  \smooth^2 T^4 d^2 \right),
\end{align*}
which completes the proof.


\section{Discussion}

We have presented an improved non-asymptotic analysis of the
Euler-Maruyama discretization of the Langevin diffusion. We have shown
that, as long as the drift term is second-order smooth, the KL
divergence between the Langevin diffusion and its discretization is
bounded as $O(\stepsize^2 d^2 T)$. Importantly, this analysis obtains
the tight $O(\stepsize)$ rate for the Euler-Maruyama scheme (under
Wasserstein or TV distances), without assuming global
contractivity. This result serves as a convenient tool for the future
study of Langevin algorithms for sampling, optimization, and
statistical inference, as it allows to directly translate
continuous-time results into discrete time, with tight rates.

Note that our results only apply to the Langevin
diffusion. Considering the discretization of more general diffusions,
either with location-varying covariance or second-order derivatives as
the underdamped Langevin dynamics~\citep{cheng2017underdamped} is a
promising direction for further research.

\subsection*{Acknowledgements}  
 This work was partially supported by Office of Naval Research Grant
 ONR-N00014-18-1-2640 to MJW and National Science Foundation Grant
 NSF-CCF-1909365 to PLB and MJW.  We also acknowledge support from
 National Science Foundation grant NSF-IIS-1619362 to PLB. We thank Xiang Cheng, Yi-An Ma and Andre Wibisono for helpful discussions.



\appendix





\section{Proofs omitted from Section~\ref{sec:kl}}
\label{app:proof}

In this section, we collect the proofs of results from
Section~\ref{sec:kl}; in particular, these results involve bounds on
the derivative $\frac{d}{dt}\kull{\discretized{\pi}_t}{\pi_t }$.


\subsection{Proof of Lemma~\ref{lemma-kl-time-derivative}}

In order to bound the derivative of the KL divergence, we first need
to interchange the order of time derivative and the integration for KL
divergence. Note that
\begin{align*}
  \frac{\partial}{\partial t}\left(\discretized{\pi}\log
  \frac{\discretized{\pi}}{\pi}\right) & = \discretized{\pi}\left(\frac{\partial \log
    \discretized{\pi}}{\partial t}(\log \discretized{\pi}+1-\log \pi) - \frac{\partial
    \log \pi}{\partial t}\right) \\
  & =\discretized{\pi}\cdot \mathrm{poly}(\nabla \log \discretized{\pi},\Delta \log
  \discretized{\pi},\nabla \log{\pi},\Delta \log \pi,b,\discretized{b}).
\end{align*}
From Lemma~\ref{lemma-coarse-estimate}, the density
$\discretized{\pi}$ has a rapidly decaying tail, and the factor grows
polynomially with $y$, uniformly in a small neighborhood of
$t$. Therefore, the integrand admits a $L^1$ integrable dominating
function over a small neighborhood of $t$. By the dominated
convergence theorem, we can exchange the order of derivative and
integration, thereby obtaining
\begin{align}
  \frac{d}{dt}\kull{\discretized{\pi} }{ \pi } =
  \int_{\real^d}\frac{\partial}{\partial t}\left(\discretized{\pi}\log
  \frac{\discretized{\pi}}{\pi}\right)dx =
  \int_{\real^d}\frac{\partial \discretized{\pi}}{\partial t}(\log
  \discretized{\pi} + 1 - \log \pi) dx- \int_{\real^d} \frac{\partial
    \pi}{\partial t}\frac{\discretized{\pi}}{\pi} dx.
\end{align}
For the first term, by Remark~\ref{remark-integration-by-parts}, we
can apply the divergence theorem, and obtain:
\begin{align*}
  \int_{\real^d} \frac{\partial \discretized{\pi}}{\partial t}(\log
  \discretized{\pi}+1-\log \pi)dx=& \int_{\real^d}\left(\nabla
  \cdot(\discretized{- \pi}\discretized{b})+\frac{1}{2}\Delta
  \discretized{\pi}\right)(\log \discretized{\pi} + 1 -\log \pi)dx
  \\
& = -\int_{\real^d} \left( - \discretized{\pi} \discretized{b} +
  \frac{1}{2} \nabla \discretized{\pi} \right) \cdot ( \nabla \log
  \discretized{\pi} - \nabla \log \pi) dx.
\end{align*}
Turning to the second term, by divergence theorem justified in
Remark~\ref{remark-integration-by-parts}, we also have:
\begin{align*}
\int_{\real^d} \frac{\partial \pi}{\partial t}
\frac{\discretized{\pi}}{\pi} dx & = \int_{\real^d} \left(- \nabla
\cdot(\pi b) + \frac{1}{2} \Delta \pi
\right)\frac{\discretized{\pi}}{\pi} dx \\
& = -\int_{\real^d}\left( - \pi b+\frac{1}{2}\nabla \pi \right)\cdot
(\nabla \log \discretized{\pi} - \nabla \log \pi) \frac{\discretized{\pi}}{\pi} dx.
\end{align*}
Putting together the pieces yields
\begin{align*}
\frac{d}{dt} \kull { \discretized{\pi}}{ \pi } & = - \int_{\real^d}
\left( - \discretized{\pi} \discretized{b} + \frac{1}{2} \nabla
\discretized{\pi}\right) \cdot (\nabla \log \discretized{\pi} - \nabla
\log \pi)-\discretized{\pi}\left(- b + \frac{1}{2}\nabla \log \pi
\right) \cdot (\nabla \log \discretized{\pi} - \nabla \log \pi) dx \\
& = \int_{\real^d} \discretized{\pi} (\nabla \log \discretized{\pi}-\nabla \log \pi)
\cdot (\discretized{b}-b) dx - \frac{1}{2} \int_{\real^d}\discretized{\pi} \Vert
\nabla \log \discretized{\pi}-\nabla \log \pi \Vert_2^2 dx \\
& \overset{(i)}{\leq} \frac{1}{2} \int_{\real^d} \discretized{\pi}
\Vert \discretized{b}-b \Vert_2^2 dx,
\end{align*}
where step (i) uses Young's inequality (namely, $u\cdot v\leq
\frac{1}{2}u^2+\frac{1}{2}v^2$).


\subsection{Proof of Lemma~\ref{lemma-remainder}}

We prove now Lemma~\ref{lemma-remainder}, which gives a bound on the
reminder term $\discretized{r}_t$ of the Taylor series
expansion~\eqref{eq:taylor}.  Let us consider the norm of
$\discretized{r}_t$ and apply the triangle inequality:
\begin{align*}
  \Vert \discretized{r}_t(x)\Vert_2 & = \left \Vert
  \mathbb{E}\left(\int_0^1 s \nabla ^2b((1-s)\discretized{X}_t+s
  \discretized{X}_{k \stepsize})[\discretized{X}_{k
      \stepsize}-\discretized{X}_t,\discretized{X}_{k
      \stepsize}-\discretized{X}_t]ds \Big|\discretized{X}_t=x
  \right)\right \Vert_2 \\
& \leq \int_0^1 \mathbb{E}\left( \opnorm{\nabla
    ^2b((1-s)\discretized{X}_t+s \discretized{X}_{k \stepsize}) }
  \cdot \Vert \discretized{X}_{k \stepsize}-\discretized{X}_t
  \Vert_2^2 \Big|\discretized{X}_t=x \right) \\
  & \leq \frac{\hessianlip}{2}\mathbb{E}\left(\Vert \discretized{X}_{k
    \stepsize}-\discretized{X}_t \Vert_2^2 \Big|\discretized{X}_t=x
  \right)\\ \leq&\hessianlip \mathbb{E}\left(\Vert (t-k
  \stepsize)b(\discretized{X}_{k \stepsize})\Vert_2^2+\left \Vert
  \int_{k \stepsize}^t dB_s \right \Vert_2^2 \Big|\discretized{X}_t=x
  \right).
\end{align*}
Taking the global expectation leads to
\begin{align*}
  \mathbb{E}\Vert \discretized{r}_t(\discretized{X}_t)\Vert_2^2 \leq& \hessianlip^2
  \mathbb{E}\mathbb{E}\left(\Vert (t-k \stepsize)b(\discretized{X}_{k
    \stepsize})\Vert_2^2+\left \Vert \int_{k \stepsize}^t dB_s \right \Vert_2^2
  \Big|\discretized{X}_t \right)^2 \\ 
  \stackrel{(i)}{\leq} &\hessianlip^2 \mathbb{E}\left(\Vert (t-k
  \stepsize)b(\discretized{X}_{k \stepsize})\Vert_2^2+\left \Vert \int_{k \stepsize}^t dB_s
  \right \Vert_2^2 \right)^2 \\ 
  \stackrel{(ii)}{\leq}& 2\hessianlip^2 \mathbb{E}\left(\Vert
  (t-k \stepsize)b(\discretized{X}_{k \stepsize})\Vert_2^4+\left \Vert \int_{k \stepsize}^t
  dB_s \right \Vert_2^4 \right)\\ \leq& 8(t-k \stepsize)^4\hessianlip^2
  \left(A_0^4+\smooth^4 \mathbb{E}\Vert \discretized{X}_{k \stepsize}\Vert_2^4
  \right)+24(t-k \stepsize)^2\hessianlip^2 d^2,
\end{align*}
where step (i) follows from the Cauchy-Schwartz inequality; and step
(ii) follows from a variant of Young's inequality (namely, $(a + b)^2
\leq 2( a^2 + b^2)$).


\section{Proofs Omitted from Section~\ref{sec:regu}}
\label{sec-appendix-regularity}

In this section, we present the proofs omitted from
Section~\ref{sec:regu}.  In particular, these results involve upper
bounds on the Fisher information and the moments.


\subsection{Proof of Lemma~\ref{lemma-mimic-forward-euler}}

We first construct an interpolated process:
\begin{align}
\discretized{X}_s' = \discretized{X}_{t_0}+ (s-t_0)
\discretized{b}_{t_0} (\discretized{X}_{t_0}) + \int_{t_0}^s
dB_t,\quad \forall s \in[ t_0, k \stepsize].
  \end{align}
According to Lemma~\ref{lemma-Fokker-Planck-for-interploated}, the
density $\discretized{\pi}'_s$ of $\discretized{X}_s'$ satisfies the following
Fokker-Planck equation:
\begin{align}
  \frac{\partial \discretized{\pi}_t'}{\partial s}(x)=\nabla \cdot
  \left(\discretized{\pi}_t'(x)\mathbb{E}\left(\discretized{b}_{t_0}(\discretized{X}_{t_0})\big|\discretized{X}_s=x
  \right)\right)+\frac{1}{2}\Delta \discretized{\pi}_t'(x).
\end{align}
Note that for $(k-1)\stepsize \leq t_0 \leq s$, we have:
\begin{align}
  \mathbb{E}\left(\discretized{b}_{t_0}(\discretized{X}_{t_0})\big|\discretized{X}_s=x
  \right)=\mathbb{E}\left(\mathbb{E}\left(b(\discretized{X}_{(k-1)\stepsize})\big|
  \discretized{X}_{t_0}\right)\Big|\discretized{X}_s=x
  \right)=\mathbb{E}\left(b(\discretized{X}_{(k-1)\stepsize})\big|\discretized{X}_s=x
  \right)=\discretized{b}_s(x).
\end{align}
Plugging back into the Fokker-Planck equation, we get:
\begin{align}
  \frac{\partial \discretized{\pi}_t'}{\partial s}(x)=\nabla \cdot
  \left(\discretized{\pi}_t'(x)\discretized{b}_{s}(x)\right)+\frac{1}{2}\Delta
  \discretized{\pi}_t'(x),
\end{align}
which is exactly the same PDE as in
Lemma~\ref{lemma-Fokker-Planck-for-interploated}. Due to the
uniqueness of solution to parabolic equations, we have:
\begin{align}
  \discretized{\pi}_s=\discretized{\pi}_s',\quad \forall s \in[t_0,k \stepsize],
\end{align}
which proves the lemma.


\subsection{Proof of Lemma~\ref{lemma-fisher-one-step-blow-up}}

For $\stepsize \in (0, \frac{1}{2\smooth})$, it is easy to see that
$\phi$ is a one-to-one mapping, and $\frac{1}{2}I \preceq \nabla \phi
\preceq \frac{3}{2}I$.
\begin{align*}
 \int p(z)\Vert \nabla_z \log p(z)\Vert_2^2dz=&\int p(\phi(x))\Vert
 \nabla_z \log p(\phi(x))\Vert_2^2 \det(\nabla \phi(x))dx \\ =&\int
 p(\phi(x))\Vert \nabla \phi(x)^{-1}\nabla_x \log p(\phi(x))\Vert_2^2
 \det(\nabla \phi(x))dx \\ =&\int \discretized{\pi}_{t_0}(x)\Vert \nabla
 \phi(x)^{-1}\left(\nabla_x \log \discretized{\pi}_{t_0}(x)-\nabla \log
 \det(\nabla \phi(x))\right)\Vert_2^2dx \\ \leq&2 \int
 \discretized{\pi}_{t_0}(x)\Vert \nabla \phi(x)^{-1}\nabla_x \log
 \discretized{\pi}_{t_0}(x)\Vert_2^2dx \\ & \quad +2 \int
 \discretized{\pi}_{t_0}(x)\Vert \nabla \phi(x)^{-1}\nabla \log \det(\nabla
 \phi(x))\Vert_2^2dx \\ \leq& 8 \int \discretized{\pi}_{t_0}(x)\Vert \nabla_x
 \log \discretized{\pi}_{t_0}(x)\Vert_2^2dx+32 \stepsize^2d^2\hessianlip^2.
\end{align*}
For the last inequality, the first term is due to $\opnorm{\nabla
  \phi(x)^{-1}} \leq 2$, and the bound for second term can be derived
as:
\begin{align*}
\int \discretized{\pi}_{t_0}(x) \Vert \nabla \phi(x)^{-1}\nabla \log
\det(\nabla \phi(x)) \Vert_2^2 dx & = \int \discretized{\pi}_{t_0}(x)\Vert
\nabla \phi(x)^{-2} \nabla \cdot(I+ (k \stepsize-t_0)\nabla
\discretized{b}_{t_0}(x)) \Vert_2^2dx \\ & \leq 16 \stepsize^2d^2 \hessianlip^2.
\end{align*}


\subsubsection{Proof of Lemma~\ref{lemma-non-increasing-fisher-heat}}

Let $\phi_t$ be the density of $d$-dimensional Gaussian with mean 0 and covariance $t I_d$. Apparently $u_t = u_0 \star \phi_t$. Note that by Cauchy-Schwartz inequality,
\begin{align*}
    \int u_t \vecnorm{\nabla \log u_t}{2}^2 dx &= \int u_0 \star \phi_t (x) \vecnorm{\frac{\int \nabla \log u_0 (x - y) u_0 (x - y)\phi_t (y) dy}{\int u_0 (x - y) \phi_t (y) dy }}{2}^2 dx\\
    &\leq \int u_0 \star \phi_t (x) \frac{\int \vecnorm{ \nabla \log u_0 (x - y)}{2}^2 u_0 (x - y)\phi_t (y) dy}{\int u_0 (x - y) \phi_t (y) dy } dx\\
    &= \int \int \vecnorm{ \nabla \log u_0 (x - y)}{2}^2 u_0 (x - y)\phi_t (y) dy dx\\
    &= \int \vecnorm{\nabla \log u_0 (z)}{2}^2 u_0 (z) \left(\int \phi_t (y) dy \right) dz\\
    &= \int \vecnorm{\nabla \log u_0 (z)}{2}^2 u_0 (z) dz,
\end{align*}
which finishes the proof.


\subsection{Proof of Lemma~\ref{lemma-tail-strong-dissipative}}
\label{sec:prooflemma-tail-strong-dissipative}

In this section, we present the proof of the moment bound given in
Lemmas~\ref{lemma-tail-strong-dissipative}.  Let the process
$\interpolated{X}_t$ be the time-inhomogeneous diffusion process
defined by generator $- \inprod{\discretized{b}_t}{\nabla} +
\frac{1}{2}\Delta$, starting from $\discretized{X}_0$. By
Lemma~\ref{lemma-Fokker-Planck-for-interploated}, the one-time
marginal laws of $\interpolated{X}_t$ and $\discretized{X}_t$ are the
same. So we only need to show the moment bounds for
$\interpolated{X}_t$.

We first apply It\^{o}'s formula, for some $c>0$, and obtain:
\begin{multline*}
\frac{1}{2} e^{ct} \Vert \interpolated{X}_t \Vert_2^2-
\frac{1}{2}\Vert \interpolated{X}_0 \Vert_2^2 =\int_0^t \langle
\interpolated{X}_s,\discretized{b}(\interpolated{X}_s)e^{cs}\rangle
ds+\frac{d}{2}\int_0^t e^{cs}ds \\
 + \int_0^t e^{cs}\interpolated{X}_s^T dB_s+\frac{1}{2}\int_0^t c
 e^{cs}\Vert \interpolated{X}_s \Vert_2^2ds.
\end{multline*}
Letting $\discretized{M}_t \defn \int_0^t \interpolated{X}_s^Te^{cs}
dB_s$ be the martingale term, using the Burkholder-Gundy-Davis
inequality~\citep{MR0400380}, for $p \geq 4$, we have:
\begin{align*}
  \mathbb{E}\sup_{0 \leq t \leq T} |M_t|^{\frac{p}{2}} & \leq (p
  C)^{\frac{p}{4}}\mathbb{E}\langle M,M \rangle_T^{\frac{p}{4}}
  \\ \leq & (p C)^{\frac{p}{4}} \mathbb{E}\left(\int_0^T e^{2cs}\Vert
  \interpolated{X}_s \Vert_2^2ds \right)^{\frac{p}{4}} \\ & \leq
  (pC)^{\frac{p}{4}}\mathbb{E}\left(\sup_{0 \leq s \leq T}e^{cs}\Vert
  \interpolated{X}_s \Vert_2^2 \cdot \int_0^T e^{cs}ds
  \right)^{\frac{p}{4}} & \leq \left(\frac{Cp
    e^{cT}}{c}\right)^{\frac{p}{4}}\left(A+\frac{1}{A}\mathbb{E}\left(\sup_{0
    \leq t \leq T}e^{ct}\Vert \interpolated{X}_t \Vert_2^2
  \right)^{\frac{p}{2}}\right).
\end{align*}
Note that this bound holds for an arbitrary value of $A > 0$; we make
a specific choice later in the argument.  On the other hand, by
Assumption~\ref{assume-strong-dissipative}, we have:
\begin{align}
  \int_0^t \langle
  \interpolated{X}_s,\discretized{b}(\interpolated{X}_s)e^{cs}\rangle
  ds \leq \int_0^t \left(-\mu \Vert \interpolated{X}_s \Vert_2^2+\beta
  \right)e^{cs}ds.
\end{align}
Putting these bounds together with $c=2 \mu$, we find that
\begin{align*}
  \mathbb{E}\left(\sup_{0 \leq t \leq T}e^{2 \mu t}\Vert
  \interpolated{X}_t \Vert_2^2 \right)^{\frac{p}{2}}\leq&
  3^{\frac{p}{2}-1}\mathbb{E}\Vert X_0
  \Vert_2^p+3^{\frac{p}{2}-1}\mathbb{E}\sup_{0 \leq t \leq T}
  |M_t|^{\frac{p}{2}}\\ &+3^{\frac{p}{2}-1} \mathbb{E}\left(\sup_{0
    \leq t \leq T}\int_0^t \left(2 \langle
  \interpolated{X}_s,\discretized{b}(\interpolated{X}_s)\rangle+d+c
  \Vert \interpolated{X}_s \Vert_2^2 \right)e^{cs}ds
  \right)^{\frac{p}{2}}\\ \leq& (C
  \sigma_0^2pd)^{\frac{p}{2}}+\left(\frac{Cp e^{2 \mu
      T}}{\mu}\right)^{\frac{p}{4}}\left(A+\frac{1}{A}\mathbb{E}\left(\sup_{0
    \leq t \leq T}e^{2 \mu t}\Vert \interpolated{X}_t \Vert_2^2
  \right)^{\frac{p}{2}}\right)\\ &+3^{\frac{p}{2}-1}
  \mathbb{E}\left(\sup_{0 \leq t \leq T}\int_0^t \left(2 \beta+d
  \right)e^{2 \mu s}ds \right)^{\frac{p}{2}},
\end{align*}
for some universal constant $C>0$.

Setting $A=2 \left(\frac{C p e^{2 \mu T}}{\mu}\right)^{\frac{p}{4}}$
and substiuting into the inequality above, we find that
\begin{align}
  \left(\mathbb{E}\Vert \interpolated{X}_T \Vert_2^p
  \right)^{\frac{1}{p}}\leq e^{-\mu T}\left(\mathbb{E}\left(\sup_{0
    \leq t \leq T}e^{2 \mu t}\Vert \discretized{X}_t \Vert_2^2
  \right)^{\frac{p}{2}}\right)^{\frac{1}{p}}\leq C'\left(e^{-\mu
    T}\sigma_0 \sqrt{pd}+\sqrt{\frac{p}{\mu}}+\sqrt{\frac{2
      \beta+d}{\mu}}\right),
\end{align}
for universal constant $C'>0$, which proves the claim.


\subsection{Proof of Lemma~\ref{lemma:gradient-sum-bound-24-nonnegative}}
\label{sec:appendix-proof-lemma-24-grad-norm}

A key technical ingredient in the proof of
Lemma~\ref{lemma:moment-24-only-non-neg} is
Lemma~\ref{lemma:gradient-sum-bound-24-nonnegative}, which bound the
second and fourth moments of gradient along the path of the
Euler-Maruyama scheme with non-negative potential functions. It is
worth noticing that the exact cancellation needed in the proof only
happens with the first and second moment of average squared gradient
norm, which is exactly what we need for the proof of
Theorem~\ref{ThmWeakAssumption}.

\begin{proof}
  For each step of the algorithm, there is:
  \begin{align*}
    f (\discretized{X}_{(k + 1) \stepsize}) - f(\discretized{X}_{k
      \stepsize}) = & f(\discretized{X}_{k \stepsize} - \stepsize
    \nabla f(\discretized{X}_{k \stepsize}) + \sqrt{\stepsize} \xi_k)
    - f(\discretized{X}_{k \stepsize})\\ \overset{(i)}{\leq} & -
    \stepsize \vecnorm{\nabla f(\discretized{X}_{k \stepsize})}{2}^2 +
    \sqrt{\stepsize} \inprod{\nabla f(\discretized{X}_{k \stepsize})
    }{\xi_k} + \frac{\smooth}{2}\vecnorm{- \stepsize \nabla
      f(\discretized{X}_{k \stepsize}) + \sqrt{\stepsize}
      \xi_k}{2}^2\\ \overset{(ii)}{\leq} & - \frac{\stepsize}{2}
    \vecnorm{\nabla f(\discretized{X}_{k \stepsize})}{2}^2 +
    \sqrt{\stepsize} (1 - \stepsize \smooth) \inprod{\nabla
      f(\discretized{X}_{k \stepsize}) }{\xi_k} + \frac{\smooth
      \stepsize}{2} \vecnorm{\xi_k}{2}^2
        \end{align*}
where step (i) follows from Assumption~\ref{assume-lipschitz-drift}
and step (ii) uses the fact that $\stepsize \leq 1 / 2\smooth$
        
Summing them together, we obtain:
\begin{align*}
  f(\discretized{X}_0) \geq f (\discretized{X}_0) -
  f(\discretized{X}_{T}) \geq \frac{\stepsize}{2} \sum_{k = 0}^{T /
    \stepsize} \vecnorm{\nabla f(\discretized{X}_{k \stepsize})}{2}^2
  - \sqrt{\stepsize} (1 - \stepsize \smooth) \sum_{k = 0}^{T /
    \stepsize} \inprod{\nabla f (\discretized{X}_{k
      \stepsize})}{\xi_k} - \frac{\smooth \stepsize}{2} \sum_{k =
    0}^{T / \stepsize} \vecnorm{\xi_k}{2}^2.
\end{align*}
Taking expectations yields
\begin{align*}
  \stepsize \sum_{k = 0}^{T / \stepsize} \Exs \vecnorm{\nabla f (\discretized{X}_{k \stepsize})}{2}^2 \leq 2 \Exs f(\discretized{X}_0) + \smooth T d. 
\end{align*}
For the higher-order moment, note that:
\begin{align*}
  \Exs \left( \sum_{k = 0}^{T / \stepsize} \inprod{\nabla f
    (\discretized{X}_{k \stepsize})}{\xi_k} \right)^2 & = \sum_{k =
    0}^{T / \stepsize} \Exs \inprod{\nabla f (\discretized{X}_{k
      \stepsize})}{\xi_k}^2 + 2 \sum_{0 \leq k < j \leq T / \stepsize}
  \Exs \left(\inprod{\nabla f (\discretized{X}_{k \stepsize})}{\xi_k}
  \cdot \inprod{\nabla f (\discretized{X}_{j \stepsize})}{\xi_j}
  \right)\\
& = \sum_{k = 0}^{T / \stepsize} \Exs \vecnorm{\nabla
    f(\discretized{X}_{k \stepsize})}{2}^2 \leq \frac{2}{\stepsize}
  \Exs f(\discretized{X}_0) + \frac{ T \smooth d}{\stepsize},
\end{align*}
where the cross term is exactly 0 for $k < j$, because
\begin{align*}
  \Exs \left(\inprod{\nabla f (\discretized{X}_{k
      \stepsize})}{\xi_k} \cdot \inprod{\nabla f
    (\discretized{X}_{j \stepsize})}{\xi_j} \right) = &
  \Exs~ \Exs \left(\inprod{\nabla f (\discretized{X}_{k
      \stepsize})}{\xi_k} \cdot \inprod{\nabla f
    (\discretized{X}_{j \stepsize})}{\xi_j} | \mathcal{F}_{j
    \stepsize}\right) \\ = & \Exs \left(\inprod{\nabla f
    (\discretized{X}_{k \stepsize})}{\xi_k} \cdot
  \inprod{\nabla f (\discretized{X}_{j \stepsize})}{ \Exs
    (\xi_j | \mathcal{F}_{j \stepsize} )} \right) = 0.
\end{align*}

So we obtain:
\begin{align*}
  \Exs \left( \stepsize \sum_{k = 0}^{T / \stepsize} \vecnorm{\nabla f
    (\discretized{X}_{k \stepsize})}{2}^2 \right)^2 \leq & 12 \Exs
  f(\discretized{X}_0)^2 + 12 \stepsize (1 - \stepsize \smooth)^2
  \left( \sum_{k = 0}^{T / \stepsize} \inprod{\nabla
    f(\discretized{X}_{k \stepsize})}{\xi_k} \right)^2 + 3 \smooth^2
  \stepsize^2 \Exs \left( \sum_{k = 0}^{T / \stepsize}
  \vecnorm{\xi_k}{2}^2 \right)^2\\ \leq & 12 \Exs
  f(\discretized{X}_0)^2 + 24 \Exs f(\discretized{X}_0) + 12 T \smooth
  d + 9 \smooth^2 T^2 d^2,
\end{align*}
which completes the proof.
\end{proof}

    
\section{Proof of Corollary~\ref{cor:main}} 

In this section, we prove the bounds on the mixing time of the
unadjusted Langevin algorithm, as stated in Corollary~\ref{cor:main}.
Recall that $\gamma(x) \propto e^{-U(x)}$ and $b(x)=-\frac{1}{2}\nabla
U(x)$.  Using these representations, we can calculate the time
derivative of $\kull{\discretized{\pi}_t}{ \gamma}$ as follows:
\begin{align*}
  \frac{d}{dt}\kull{\discretized{\pi}_t}{\gamma} & = \int
  \frac{\partial \discretized{\pi}_t}{\partial t}(1+\log
  \discretized{\pi}_t-\log \gamma)\\
& = \int \left( - \nabla \cdot(\discretized{\pi}_t
  \discretized{b}_t)+\frac{1}{2}\Delta \discretized{\pi}_t
  \right)(1+\log \discretized{\pi}_t-\log \gamma) \\
& \stackrel{(i)}{=}  \int \left(\discretized{\pi}_t
  \discretized{b}_t - \frac{1}{2}\nabla \discretized{\pi}_t \right)^T
  (\nabla \log \discretized{\pi}_t-\nabla \log \gamma) \\ & = -
  \frac{1}{2}\int \discretized{\pi}_t \Vert \nabla \log
  \discretized{\pi}_t-\nabla \log \gamma \Vert_2^2 - \int
  \discretized{\pi}_t \langle \nabla \log \discretized{\pi}_t-\nabla
  \log \gamma, \discretized{b}+\frac{1}{2}\nabla U \rangle\\
 & \leq -\frac{1}{2}\int \discretized{\pi}_t \Vert \nabla \log
  \discretized{\pi}_t-\nabla \log \gamma \Vert_2^2 + \int
  \discretized{\pi}_t \Vert \nabla \log \discretized{\pi}_t-\nabla
  \log \gamma \Vert_2 \cdot \Vert \discretized{b}+\frac{1}{2}\nabla U
  \Vert_2 \\
  & \stackrel{(ii)}{\leq} - \frac{1}{4}\int \discretized{\pi}_t \Vert
  \nabla \log \discretized{\pi}_t-\nabla \log \gamma \Vert_2^2+\int
  \discretized{\pi}_t \Vert \discretized{b}_t(x) - b(x) \Vert_2^2 dx,
\end{align*}
where step (i) follows from the divergence theorem, and step (ii)
follows from Young's inequality (that is, $a b \leq a^2 / 4 +
b^2$).  
 Step (i) justified by Lemma~\ref{lemma-green-formula} in
Appendix~\ref{sect:appendix-coarse-estimates}, where the exponential
tail condition directly follows Lemma~\ref{lemma-coarse-estimate}.

The first term gives a minus KL term based on log-Sobolev inequality,
whereas the second term corresponds to what we have estimated in
previous sections. Using same type of analysis, we find that
\begin{multline*}
  \kull{\discretized{\pi}_T}{\gamma} \leq e^{-\frac{1}{4}\rho
    T}\kull{\discretized{\pi}_0}{\gamma} +
  \stepsize^4\hessianlip^2(d^2+A_0^4+\smooth^4(\sigma^2d+\frac{1}{\mu}+\beta)^2)
  \\
  + \stepsize^2 \left(h_0 + A_0^2 + T( \sigma_0^2 d + \frac{1}{\mu} +
  \beta) (\sigma^{-2} + \smooth^2) + \hessianlip^2 d^2 \right),
\end{multline*}
where $\rho$ is the log-Sobolev constant. Regarding all the smoothness
parameters as constants, and using the codnition
$\kull{\discretized{\pi}}{\gamma}\leq \varepsilon$, we find that
\begin{align}
 N = \tilde{O} \left( \frac{d}{\rho} \left(\frac{1}{\rho
   \varepsilon}\right)^{\frac{1}{2}} \right).
\end{align}
In order to translate this result into TV distance, we simply apply
Pinsker's inequality. In order to otbtain Wasserstein distance bound,
we can use the Talagrand transportation
inequality~\citep{talagrand1991new,otto2000generalization}.
\begin{align*}
 \Wass_2 ( \discretized{\pi}_T, \gamma ) \leq \sqrt{ \frac{2}{\rho}
   \kull { \discretized{\pi}_T}{ \gamma }}.
\end{align*}
Since the log-Sobolev constant is potentially large, we can also use
the weighted Csisz\'{a}r-Kullback-Pinsker inequality
of~\citet{bolley2005weighted} to relate the KL divergence to the
Wasserstein $\Wass_1$ distance.  In particular, for any choice of
$\gamma$ such that
\begin{subequations}
\begin{align}
  \label{eq:cstvil}
C_\gamma \defn 2 \inf_{\alpha > 0 } \left( \frac{1}{2\alpha} \left( 1
+ \log \int e^{ \alpha \Vert x\Vert_2^2}
d\gamma(x)\right)\right)^{\frac{1}{2}} < +\infty
\end{align}
is finite, we have
\begin{align}
  \Wass_1 ( \discretized{\pi}_T, \gamma)\leq C_\gamma
  \sqrt{\kull { \discretized{\pi}_T}{ \gamma }},
\end{align}
\end{subequations}

We compute here an upper bound on the constant $C_\gamma$ for
stationary distribution $\gamma$. Note that the proof of
Lemma~\ref{lemma-tail-strong-dissipative} in
Section~\ref{sec:prooflemma-tail-strong-dissipative} also goes through
for the Langevin diffusion itself, which converges asymptotically to
$\gamma$. Therefore adapting Lemma~\ref{lemma-tail-strong-dissipative}
to the Langevin process $X_t$ we obtain for any $p>1$:
\begin{align*}
\int_{\real^d} \Vert x\Vert_2^p d \gamma (x) = \limsup_{T \rightarrow
  +\infty } \mathbb{E} \Vert X_t\Vert_2^p \leq \left( C'\frac{p +
  \beta + d}{\mu}\right)^{\frac{p}{2}}.
\end{align*}
Expanding $e^{ \alpha \Vert x\Vert_2^2}$ for $\alpha = \frac{ \mu }{ 8
  C' e}$ in a Taylor series, we find that
\begin{align*}
    \int e^{ \alpha \Vert x\Vert_2^2} d\gamma(x) \leq & 1+
    \sum_{p=1}^{+\infty} \frac{1}{p!} \int_{\real^d} \alpha^p \Vert
    x\Vert_2^{2p} d \gamma (x) \\ \stackrel{(i)}{\leq }& 1 +
    \sum_{p=1}^{+\infty} \frac{1}{p!} \left( \frac{2p + \beta + d}{ 8
      e } \right)^p \leq 1 + \sum_{ p = 1}^{+\infty} \frac{1}{p!}
    \left( \frac{ p }{ 2 e } \right)^p + \sum_{ p = 1}^{+\infty}
    \frac{1}{p!} \left( \frac{\beta + d }{ 4 e } \right)^p
    \\ \stackrel{ (ii) }{ \leq}& 1 + \sum_{ p = 1}^{+\infty}
    \frac{1}{\sqrt{2 \pi p}} \left( \frac{e}{p} \right)^p \left(
    \frac{ p }{ 2 e } \right)^p + \sum_{ p = 1}^{+\infty} \frac{1}{p!}
    \left( \frac{\beta + d }{ 4 e } \right)^p \\ \leq &1+
    \sum_{p=1}^{+\infty} \left( 2^{-p} + \frac{1}{p!} (\beta + d)^p
    \right) = 1+ e^{ \beta + d}.
\end{align*}
In obtaining step (i), we plug in the moment estimate for $X$ under
$\gamma$, and in step (ii), we use the Stirling's lower bound on the
factorial function.  Plugging into equation~\eqref{eq:cstvil}, we
obtain $C_\gamma \leq C' \sqrt{ \frac{ \beta + d }{ \mu }}$ for some
universal constant $C'>0$.


\section{Coarse Tail and Smoothness Control}
\label{sect:appendix-coarse-estimates}

First, we state and prove a lemma that gives bounds on the behavior of
the densities $\pi$ and $\discretized{\pi}$ defined by the
Fokker-Planck equations~\eqref{eq:fokker_langevin}
and~\eqref{eq:fokker_inter}, respectively.

\begin{lemma}
\label{lemma-coarse-estimate}
For the densities defined by Fokker-Planck
equations~\eqref{eq:fokker_langevin} and~\eqref{eq:fokker_inter}, the
following bounds hold for all $x \in \real^d$:
\begin{subequations}
  \begin{align}
    \label{EqnDensityBound}
  \max(\discretized{\pi}(x),\pi(x)) &\leq A e^{-r \Vert x \Vert_2^2},
  \\
  \label{EqnScoreBound}
  \max \left \{ \Vert \nabla \log \discretized{\pi}(x) \Vert_2, \Vert
  \nabla \log \pi(x) \Vert_2 \right \} & \leq C (1 + \Vert x \Vert_2),
  \quad \mbox{and} \\
  \label{EqnFisherBound}
\max \left \{ \opnorm{ \nabla^2 \log \discretized{\pi}(x)} , \opnorm{
  \nabla^2 \log \pi(x)}\right \} & \leq C (1 + \Vert x \Vert_2^2).
\end{align}
\end{subequations}
Here the constants $(A, r, C)$ are independent of $x$, uniform in an
arbitrarily small neighborhood of $t$, but may dependent on other
parameters of the diffusions.
\end{lemma}
\noindent We split the proof into different parts, corresponding to
the different bounds claimed.

\paragraph{Proof of equation~\eqref{EqnDensityBound}:}

The claim for the Fokker-Planck equation of the original SDE is a
classical result~\citep[see, e.g.,][]{Pav14}. For the density
$\discretized{\pi}_t(x)$, we exploit the properties of the underlying
discrete-time update from which it arose; in particular, we prove the
claimed bound~\eqref{EqnDensityBound} via induction on the index $k$.
For $k=0$, the density $\discretized{\pi}_0$ satisfies the claimed
bound by Assumption~\ref{assume-smooth-initialize}.

Suppose that the bound~\eqref{EqnDensityBound} holds for $t=k
\stepsize$; we need to prove that it also holds for any any $t \in [k
  \stepsize,(k+1)\stepsize]$. Given this value of $t$ fixed, by the
definition of interpolated process, $\discretized{\pi}_t$ is the
consequence of push-forward measure under a non-linear transformation
on $\discretized{\pi}_{k \stepsize}$, with a Gaussian noise convoluted
with it. Suppose $Z=\phi(X_{k \stepsize})\defn \discretized{X}_{k
  \stepsize}+(t-k \stepsize)b(\discretized{X}_{k \stepsize})$ and let
$p(\cdot)$ be its density, and suppose $\stepsize<\frac{1}{2\smooth}$,
by the change of variable formula, we have
\begin{align}
    p(z) = \frac{\discretized{\pi}_{k \stepsize} (\phi^{-1}(z))}{\det
      \left((\nabla \phi)(\phi^{-1}(z))\right)}.\label{eq:defp}
\end{align}
Using Assumption~\ref{assume-lipschitz-drift} and the induction
hypothesis, we obtain the following rough bounds:
\begin{align*}
  &\discretized{\pi}_{k \stepsize}(\phi^{-1}(z))\leq Ae^{-r \Vert
    \phi^{-1}(z)\Vert_2^2}\leq Ae^{-r \Vert 3z/2 \Vert_2^2},\\ &\det
  \left((\nabla \phi)(\phi^{-1}(z))\right)\geq \frac{1}{2^d}.
\end{align*}
Therefore the density $p$ also satisfies~\eqref{EqnDensityBound} (The
constant can increase with time, but we only need it to be bounded for
finite number of steps and do not require explicit bound.)

The convolution with the Gaussian density is easy to control. Let
$\mu$ be the density of $\mathcal{N}(0,t-k \stepsize)$, we have:
\begin{align*}
  \discretized{\pi}_t(x)=p*\mu(x)=\int p(y)\mu(x-y)dy \leq \int Ae^{-r
    \Vert y \Vert_2^2}(2 \pi(t-k
  \stepsize))^{-\frac{d}{2}}e^{-\frac{\Vert x-y \Vert_2^2}{2(t-k
      \stepsize)}}dy \leq A'e^{-r'\Vert x \Vert_2^2},
\end{align*}
which completes the inductive proof.

\paragraph{Proof of the bound~\eqref{EqnScoreBound}:} The claim for $\pi_t$ is a
classical result~\citep[see, e.g.,][]{Pav14}. It remains to prove the
result for the interpolated process.  As before, we proceed via
induction.

Beginning with the base case $k = 0$, assume the result holds true for
$\discretized{\pi}_0$.  Defining $\phi$, $Z$ and $p$ as in the proof
of equation~\eqref{EqnDensityBound}, we observe that
\begin{align*}
\nabla \log p(z) =\nabla_z \log \discretized{\pi}_{k
  \stepsize}(\phi^{-1}(z))-\nabla_z \log \det \left((\nabla
\phi)(\phi^{-1}(z))\right).
\end{align*}
The first term can be controlled easily using induction hypothesis:
\begin{align*}
  \Vert \nabla_z \log \discretized{\pi}_{k \stepsize} ( \phi^{-1} (z)
  ) \Vert_2 \leq \Vert \nabla_z \phi^{-1} (z) \Vert_2 \cdot C(1 +
  \Vert \phi^{-1} (x) \Vert_2 ) \leq 2C (1 + 2 \Vert x \Vert_2 ).
\end{align*}
For the second term, note that:
\begin{align*}
   \Vert \nabla_z \log \det \left((\nabla \phi)(\phi^{-1}(z))\right)
   \Vert_2 \leq d \opnorm{ (\nabla \phi (\phi^{-1}(z)) )^{-1}} \cdot
   \Big(\sum_{i=1}^d \opnorm{ \frac{\partial }{\partial z_i} \nabla
     \phi (\phi^{-1 } (z)) }\Big) \leq 4d^2 \hessianlip,
\end{align*}
where we used Assumption~\ref{assume-lipschitz-drift} and
Assumption~\ref{assume-smooth-drift}.

For convolution with Gaussian density, we have:
\begin{align*}
  \Vert \nabla \log(p*\mu)(x)\Vert_2 = \left \Vert \nabla_x \log \int
  p(x-y) \mu(y)dy \right \Vert_2 &\leq \frac{\int \Vert \nabla \log
    p(x-y) \Vert_2 p(x-y) \mu(y) dy}{\int p(x-y) \mu(y)dy} \\ & \leq 2
  C \frac{\int (1+\Vert y \Vert_2) p(y)\mu(x-y)dy}{\int p(y)
    \mu(x-y)dy}.
  \end{align*}
  
As we have shown $p(y)$ decays with $e^{-r\Vert y\Vert_2^2}$ as $\Vert
y\Vert_2\rightarrow +\infty$. And for this fixed density function,
there exists a constant $A>0$ such that $p(\mathcal{B}(0,A)) \geq
\frac{1}{2}$. So we have:
\begin{align*}
  \int_{\mathcal{B} (0, A)}  p(y) \mu( x - y) dy \geq C \exp( - (\Vert x\Vert_2^2 +A^2) ).
\end{align*}
Since both $p$ and $\mu$ have tail decaying with $e^{- r \Vert
  x\Vert_2^2}$, there exists $K > 0$, such that:
\begin{align*}
  \int_{\mathcal{B} (0, nK (A + \Vert x\Vert_2))^C} p(y)\mu(x-y)dy
  \leq \exp( - (\Vert x\Vert_2^2 + A^2 + n)).
\end{align*}
Putting them together, we obtain:
\begin{align*}
 & \frac{\int (1+\Vert y \Vert_2) p(y)\mu(x-y)dy}{\int
    p(y)\mu(x-y)dy}\\ \leq & K ( A+ \Vert x\Vert_2 ) + \sum_{ n = 1}^{
    + \infty} C e^{ \Vert x\Vert_2^2 +A^2} (n+1)K (A + \Vert x\Vert_2)
  \int_{\mathcal{B} (0, nK (A + \Vert x\Vert_2))^C} p(y)\mu(x-y)dy
  \\ \leq & C_1 + C_2 \Vert x \Vert_2,
\end{align*}
which finishes the induction proof.


\paragraph{Proof of the bound~\eqref{EqnFisherBound}:}

The result of the Fokker-Planck equation for the original SDE is known
as well~\cite{Pav14}. Now we show the result for the interpolated
process. Once again, we proceed by induction. The result for the
initial distribution is assumed to be true. As in the proof of the
bound~\eqref{EqnDensityBound}, we consider the auxiliary variable $Z$
with the density function $p$ defined in \eqref{eq:defp}. Taking one
more derivative, we obtain:
\begin{align*}
  \nabla^2 \log p(z) =& \nabla^2_z \log \discretized{\pi}_{k
    \stepsize}(\phi^{-1}(z))-\nabla^2_z \log \det \left((\nabla
  \phi)(\phi^{-1}(z))\right)\\ = & (\nabla_z \nabla_z \phi^{-1} (z))
  \cdot \nabla_z \log \discretized{\pi}_{k\stepsize} + \nabla_z
  \phi^{-1} (z) \nabla^2_z \log \discretized{\pi}_{k \stepsize} \\ &-
  \nabla_z (\nabla \phi (\phi^{-1} (z)))^{-1}\cdot \nabla_z (\nabla
  \phi (\phi^{-1} (z))) - (\nabla \phi (\phi^{-1} (z)))^{-1} \cdot
  \nabla_z^2 (\nabla \phi (\phi^{-1} (z))).
\end{align*}
In order to bound $\opnorm{\nabla^2 \log p(z)}$, we directly control
the tensor norm of each tensor appearing in the above expression. We
bound the first and second terms using the induction hypothesis, which
leads to terms that grow linearly with $\Vert x\Vert$. The rest of the
terms are controlled simply by uniform upper bounds on the
higher-order smoothness of $b$, at a price of additional dimension
factors.

For convolution with Gaussian, note that:
\begin{align*}
  \opnorm{ \nabla_z^2 \log (p * \mu) (x) } = \opnorm{ \frac{
      \nabla^2(p * \mu) } { p * \mu } (x) } + \opnorm{\frac{(\nabla
      p* \mu )(\nabla p* \mu )^T}{ (p * \mu)^2 } (x)}.
\end{align*}
The second term is actually $\Vert \frac{\nabla p* \mu }{ p * \mu
}\Vert_2^2$, so it is already controlled by $C (1 + \Vert x
\Vert_2^2)$ ( following~\eqref{EqnScoreBound}). For the first term,
note that:
\begin{align*}
    \opnorm{ \frac{ \nabla^2(p * \mu) } { p * \mu } (x) } = &\opnorm{
      \frac{ \int \nabla \log p (y) \nabla \log \mu (x - y) p (y) \mu
        ( x - y) dy} { p * \mu (x)} }\\ \leq & C \frac{\int (1 + \Vert
      x\Vert_2^2+\Vert y \Vert_2^2) p(y)\mu(x-y)dy}{\int
      p(y)\mu(x-y)dy}.
\end{align*}
Since we have shown that the tail of $p(x)$ decays as $e^{-r \Vert
  x\Vert_2^2}$, the same argument as in the proof
of~\eqref{EqnFisherBound} also holds for this integral, and the term
$\opnorm{ \frac{ \nabla^2(p * \mu) } { p * \mu } (x) }$ is upper
bounded by $C ( 1 + \Vert x\Vert_2^2)$ for constant $C$ independent of
$x$. This finishes the induction proof.

Lemma~\ref{lemma-coarse-estimate} can be combined with the following
lemma to justify the Green formula used throughout the paper:
\begin{lemma}
  \label{lemma-green-formula}
  For functions $f, g: \real^d\rightarrow \real$ with continuous
  gradients, if there exists constants $C, r>0$ such that $\max (|
  f(x) g(x) |, \Vert \nabla f(x) g(x)\Vert_2 , \Vert f(x) \nabla
  g(x)\Vert_2 )\leq C e^{- r \Vert x\Vert_2^2}$, we have:
  \begin{align*}
    \int_{\real^d} f(x) \nabla g(x) dx = -\int_{\real^d}
    \nabla f(x) g(x) dx.
    \end{align*}
\end{lemma}
\begin{proof}
 The integratability is justified by the tail assumption on $g\nabla f
 $ and $f \nabla g$. For any $R>0$, using the Green formula for
 bounded set, we have:
 \begin{align*}
   \int_{\mathcal{B} (0, R) } f(x) \nabla g(x) dx =
   \oint_{\partial \mathcal{B} (0, R)} f(x) g(x) \nu dS -
   \int_{\mathcal{B} (0, R) } f(x) \nabla g(x) dx,
 \end{align*}
 where $\nu$ is the normal vector of the boundary at $x$. Note that:
\begin{align*}
  \Big\Vert \oint_{\partial \mathcal{B} (0, R)} f(x) g(x) \nu dS
  \Big\Vert_2 \leq \oint_{\partial \mathcal{B} (0, R)} |f(x) g(x)|
  dS \leq C e^{-r R^2} C' R^{d-1}.
\end{align*}
Letting $R \rightarrow +\infty$, we obtain:
\begin{align*}
  \int_{\real^d} f(x) \nabla g(x) dx =& \lim_{R \rightarrow
    +\infty} \int_{\mathcal{B} (0, R) } f(x) \nabla g(x) dx \\ =&
  \lim_{R \rightarrow +\infty} \oint_{\partial \mathcal{B} (0, R)}
  f(x) g(x) \nu dS - \lim_{R \rightarrow +\infty} \int_{\mathcal{B}
    (0, R) } f(x) \nabla g(x) dx \\ =& -\int_{\real^d} \nabla
  f(x) g(x) dx.
    \end{align*}
\end{proof}
\begin{remark}
  \label{remark-integration-by-parts}
Combining Lemma~\ref{lemma-coarse-estimate} and
Lemma~\ref{lemma-green-formula}, we can justify the divergence
theorems used throughout this paper. Specifically, when one of the
functions $f$ and $g$ is in the form $\pi \cdot\mathrm{poly} (\nabla
\log \pi, \nabla \log \discretized{\pi}, b, \discretized{b})$ or
$\discretized{\pi} \cdot \mathrm{poly}(\nabla \log \pi, \nabla \log
\discretized{\pi}, b, \discretized{b})$, while the other one is of the
form $\mathrm{poly}(\nabla \log \pi, \nabla \log \discretized{\pi}, b,
\discretized{b})$, the integration by parts can go through.
\end{remark}




\bibliographystyle{plainnat} \bibliography{LangevinKL}

\end{document}

%% file: final_macros.tex
\newcommand{\real}{\ensuremath{\mathbb{R}}}

\newcommand{\dist}{\mathrm{dist}}

\newcommand{\hessianlip}{L_2}

\newcommand{\discretized}{\hat}
\newcommand{\interpolated}{\tilde}

\newcommand{\stepsize}{\eta}

\newcommand{\unicon}{c}

\newcommand{\smooth}{L_1}

\newcommand{\defn}{:=}

\newcommand{\matsnorm}[2]{|\!|\!| #1 | \! | \!|_{{#2}}}
\newcommand{\vecnorm}[2]{\left\| #1\right\|_{#2}}

\newcommand{\opnorm}[1]{\ensuremath{\matsnorm{#1}{\tiny{\mbox{op}}}}}

\newcommand{\inprod}[2]{\ensuremath{\langle #1 , \, #2 \rangle}}

\newcommand{\kull}[2]{\ensuremath{D_{\text{KL}}(#1 \| #2)}}

\newcommand{\Exs}{\ensuremath{{\mathbb{E}}}}

\newtheoremstyle{named}{}{}{\itshape}{}{\bfseries}{.}{.5em}{\thmnote{#3's }#1}
\theoremstyle{named}

\theoremstyle{plain}

\newtheorem{theorem}{Theorem}
\newtheorem{proposition}{Proposition}
\newtheorem{lemma}{Lemma}

\newtheorem{corollary}{Corollary}

\newtheorem{remark}{Remark}
\newlength{\widebarargwidth}
\newlength{\widebarargheight}
\newlength{\widebarargdepth}

\makeatletter
\long\def\@makecaption#1#2{
        \vskip 0.8ex
        \setbox\@tempboxa\hbox{\small {\bf #1:} #2}
        \parindent 1.5em  
        \dimen0=\hsize
        \advance\dimen0 by -3em
        \ifdim \wd\@tempboxa >\dimen0
                \hbox to \hsize{
                        \parindent 0em
                        \hfil
                        \parbox{\dimen0}{\def\baselinestretch{0.96}\small
                                {\bf #1.} #2
                                }
                        \hfil}
        \else \hbox to \hsize{\hfil \box\@tempboxa \hfil}
        \fi
        }
\makeatother

\long\def\comment#1{}
\definecolor{battleshipgrey}{rgb}{0.52, 0.52, 0.51}
\definecolor{darkgray}{rgb}{0.66, 0.66, 0.66}
\definecolor{darkgreen}{rgb}{0.0, 0.2, 0.13}
\definecolor{darkspringgreen}{rgb}{0.09, 0.45, 0.27}
\definecolor{dukeblue}{rgb}{0.0, 0.0, 0.61}
\definecolor{olivedrab7}{rgb}{0.24, 0.2, 0.12}
\definecolor{darkblue}{rgb}{0.0, 0.0, 0.55}
\definecolor{darkscarlet}{rgb}{0.34, 0.01, 0.1}
\definecolor{candyapplered}{rgb}{1.0, 0.03, 0.0}
\definecolor{ao(english)}{rgb}{0.0, 0.5, 0.0}
\definecolor{applegreen}{rgb}{0.55, 0.71, 0.0}
